%% file: counting_DLT-Histories_arxiv.tex
\documentclass[envcountsect]{svjour3}

\usepackage[utf8]{inputenc}
\usepackage{fullpage}
\usepackage{amsmath,amssymb} 
\usepackage{cite} 
\usepackage{mathscinet}
\usepackage{bbm}
\usepackage{booktabs}
\usepackage{subfig}
\usepackage[T1]{fontenc}
\usepackage{csvsimple}
\usepackage{graphicx}

\newcommand{\LandauO}{\mathcal{O}}

\newcommand{\Dup}{\mathbb{D}}
\newcommand{\Loss}{\mathbb{L}}
\newcommand{\Trans}{\mathbb{T}}
\newcommand{\Spec}{\mathbb{S}}
\newcommand{\Extant}{Extant}

\newcommand{\DL}{\ensuremath{\Dup\Loss}}
\newcommand{\DLT}{\ensuremath{\Dup\Loss\Trans}}

\newcommand{\Unranked}{\text{\sffamily u}}
\newcommand{\Ranked}{\text{\sffamily r}}

\newcommand{\UDL}{\ensuremath{\Unranked\Dup\Loss}}
\newcommand{\UDLT}{\ensuremath{\Unranked\Dup\Loss\Trans}}
\newcommand{\RDL}{\ensuremath{\Ranked\Dup\Loss}}
\newcommand{\RDLT}{\ensuremath{\Ranked\Dup\Loss\Trans}}
\newcommand{\SL}{\ensuremath{\Spec\Loss}}
\newcommand{\RDTSL}{\ensuremath{\Ranked\Dup\Trans}-\SL}

\newcommand{\Zc}{\mathcal{Z}}
\newcommand{\Xc}{\mathcal{X}}
\newcommand{\Yc}{\mathcal{Y}}
\newcommand{\Wc}{\mathcal{W}}
\newcommand{\Sc}{\mathcal{S}}

\newcommand{\Bc}{{B}}
\newcommand{\Hc}{{H}}
\newcommand{\Dc}{{D}}
\newcommand{\Tc}{{T}}

\newcommand{\Sp}{\mathbf{S}}
\newcommand{\T}{\mathbf{T}}
\newcommand{\CT}{\mathbf{CT}}
\newcommand{\CB}{\mathbf{CB}}

\newcommand{\CC}{\mathbb{C}}

\usepackage{xcolor}
\usepackage{hyperref}
\definecolor{darkgreen}{rgb}{0,0.4,0}
\definecolor{BrickRed}{rgb}{0.65,0.08,0}
\hypersetup{colorlinks=true,linkcolor=blue,citecolor=darkgreen,filecolor=BrickRed,urlcolor=darkgreen}

\newcommand{\OEIS}[1]{\href{http://oeis.org/#1}{OEIS~#1}}
\newcommand{\OEISs}[1]{\href{http://oeis.org/#1}{#1}}

\newcommand{\addorcid}[1]{\protect\includegraphics[height=3mm]{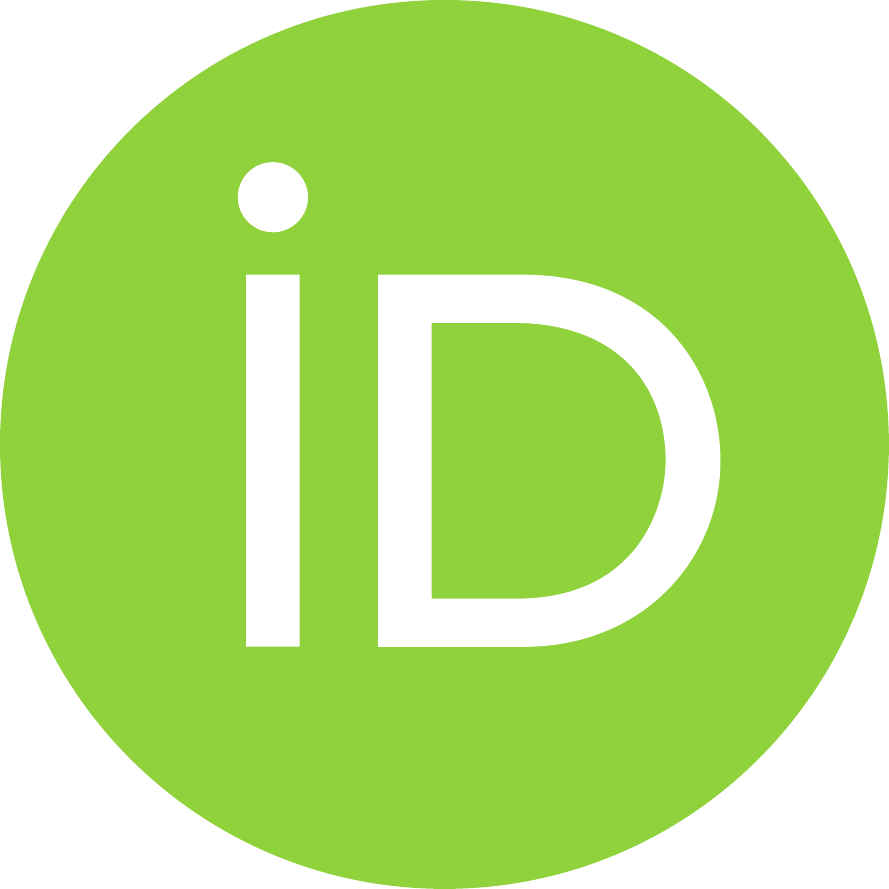} \href{https://orcid.org/#1}{#1}}

\newcommand{\exampleend}{\hfill $\blacksquare$}

\begin{document}

\title{Counting and sampling gene family evolutionary histories in the duplication-loss and duplication-loss-transfer models}
\date{}
\titlerunning{Counting and sampling gene family evolutionary histories}  
%
\author{Cedric Chauve
	\and Yann Ponty
    \and Michael Wallner
    }

\institute{Cedric Chauve \at Department of Mathematics, Simon Fraser University, Burnaby (BC), Canada\\ LaBRI, Universit\'e de Bordeaux, Talence, France\\ LIX, Ecole Polytechnique, Palaiseau, France\\ \addorcid{0000-0001-9837-1878}
    \and Yann Ponty \at CNRS and LIX, Ecole Polytechnique, Palaiseau, France\\
    \addorcid{0000-0002-7615-3930}
	\and Michael Wallner \at LaBRI, Universit\'e de Bordeaux, Talence, France\\
	Institut f\"ur Diskrete Mathematik und Geometrie, TU Wien, Vienna, Austria\\
	\addorcid{0000-0001-8581-449X}}

\maketitle              

\begin{abstract}
Given a set of species whose evolution is represented by a species tree, a gene family is a group of  genes having evolved from a single ancestral gene. A gene family evolves along the branches of a species tree through various mechanisms, including -- but not limited to -- speciation ($\Spec$), gene duplication ($\Dup$), gene loss ($\Loss$), horizontal gene transfer ($\Trans$).
The reconstruction of a gene tree representing the evolution of a  gene family constrained by a  species tree is an important problem in phylogenomics. However, unlike in the multispecies coalescent evolutionary model that considers only speciation and incomplete lineage sorting events, very little is known about the search space for gene family histories accounting for gene duplication, gene loss and horizontal gene transfer (the \DLT-model). 

In this work, we introduce the notion of evolutionary histories  defined as a binary ordered rooted tree describing the evolution of a gene family, constrained by a species tree in the \DLT-model. We provide formal grammars describing the set of all evolutionary histories that are compatible with a given species tree, whether it is ranked or unranked. These grammars allow us, using either analytic combinatorics or dynamic programming, to efficiently compute  the number of histories of a given size, and also to generate random histories of a given size under the uniform distribution. We apply these tools to obtain exact asymptotics for the number of gene family histories for two species trees, the rooted caterpillar and the complete binary tree, as well as estimates of the range of the exponential growth factor of the number of histories for random species trees of size up to $25$. Our results show that including horizontal gene transfer induce a dramatic increase of the number of evolutionary histories. We also show that, within ranked species trees, the number of evolutionary histories in the \DLT-model is almost independent of the species tree topology. These results establish firm foundations for the development of ensemble methods for the prediction of  reconciliations.

\noindent

\noindent\subclassname{92B99,05A15,05A16}
\end{abstract}

\input{intro.tex}

\input{model.tex}

\input{methods.tex}
\input{results.tex}

\input{conclusion.tex}

\bibliographystyle{abbrv}
\bibliography{references}

\end{document}

%% file: intro.tex
\section{Introduction}
\label{sec:intro}

A gene tree represents the evolution of a gene family, a group of genes assumed to descend from a single ancestral gene. The reconstruction of gene trees from molecular sequence data is a central but difficult problem in computational biology. Indeed, while species are mostly expected to evolve through \textit{speciation}, gene families evolve through a wider variety of mechanisms including gene duplication, gene loss, horizontal gene transfer (HGT) and incomplete lineage sorting (ILS). As a result, it is common to observe an incongruence between gene trees and species trees~\cite{sysbio/Maddison97}. This discrepancy has motivated an intense research activity on the problem of reconstructing the gene tree of a gene family, conditional to a given species tree for the considered species. We refer to~\cite{mmb/SzollosiD12,sysbio/SzollosiTDB15} for extensive reviews discussing how gene trees evolve within a species tree, describe existing models and methods for reconstructing gene trees within species trees. 

In the case where a gene family contains a single gene per species, observed incongruences between a gene tree and a species tree can be analyzed through the prism of ILS in the \textit{multispecies coalescent model}~\cite{tree/DegnanR09}. The natural question is then to compute the probability of  \textit{coalescent histories} conditional to the given species tree~\cite{evolution/DegnanS05,evolution/Wu12,bioinformatics/Wu16,bioinformatics/Wu17}. For gene families that might contain duplicate copies (or no copy) of a gene in a given species, the multispecies coalescent model is not appropriate, and gene trees need to be inferred in a model including gene duplication, gene loss and, ideally, transfers. Most methods developed to understand the evolution of gene families in this context rely on the concept of  \textit{gene tree-species tree reconciliation}, illustrated in Fig.~\ref{fig:histories}. In this framework, given a gene tree $G$ and a species tree $S$, one aims to embed $G$ within  $S$, often optimizing a parsimony or probabilistic criterion with regard to the considered evolutionary model. 

\begin{figure}[htbp]
    {\centering
    \includegraphics[width=\textwidth]{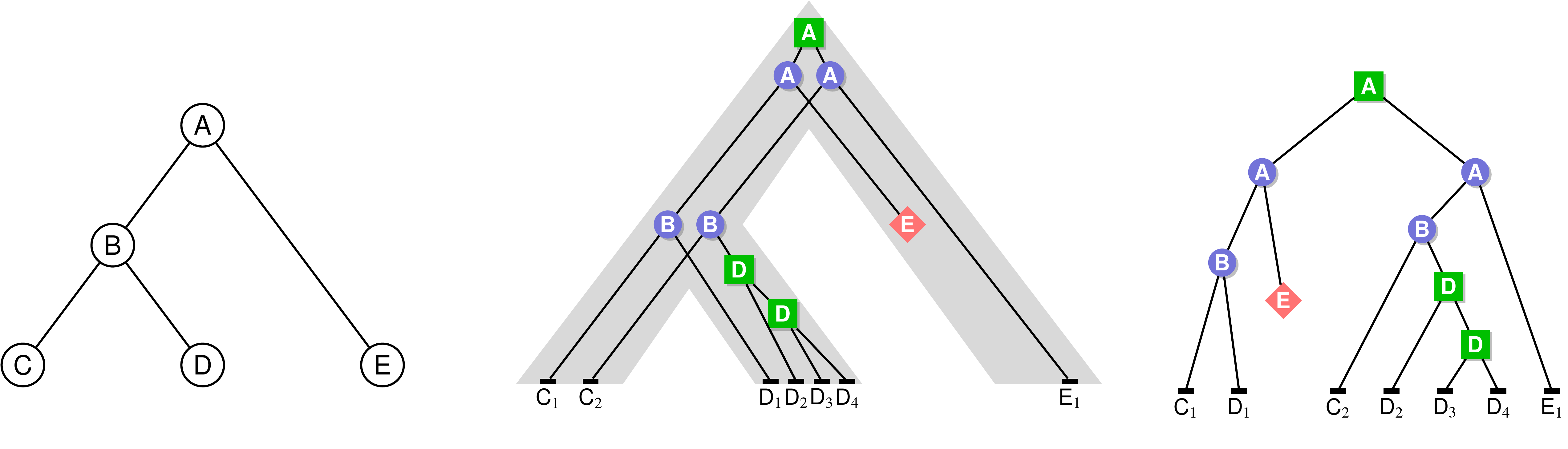}\\}
    
    \caption{A species tree $\Sp$ (left), a \DL-history for $\Sp$ (center) and its associated gene tree (right). Green squares (resp. blue circles, red diamonds, black rectangles) correspond to nodes $x$ such that $e(x)=\Dup$ (resp. $e(x)=\Spec,e(x)=\Loss,e(x)=\Extant$). The mapping $s$ is represented by the location of the internal nodes of the history within the species tree in the center tree and by the species names in the nodes in the right tree.}
    \label{fig:histories}
\end{figure}

Early reconciliation methods were developed for an evolutionary model considering only gene duplications and gene losses (the \DL-model), and considered a parsimony criterion. This problem, introduced by Goodman \emph{et al.}~\cite{syszoo/GoodmanCMRM79}, is computationally tractable through dynamic programming. Extending the model to include HGT, while ensuring that HGT events are time-consistent, makes the problem of predicting of the most parsimonious reconciliation intractable in general~\cite{cmb/OvadiaFCL11,tcbb/TofighHL11}. However, if the provided species tree is \textit{ranked}, i.e. is provided with a total ordering of its internal nodes describing the order of speciation events, the reconciliation problem becomes tractable (see the discussion in~\cite{bib/DoyonRDB11}). Over the last 20 years, various efficient dynamic programming algorithms were designed to compute a parsimonious reconciliation, implemented in widely used phylogenomics packages~\cite{cmb/DurandHV06,bioinformatics/BansalKKK18,bioinformatics/ScornavaccaJS15,bioinformatics/JacoxCSPS16}. Similar to parsimony-based methods, probabilistic reconciliation methods were first developed in a model considering only gene duplication and gene loss~\cite{jacm/ArvestadLS09,pnas/AkerborgSAL09,bmcbi/GoreckiBE11,cmb/GoreckiE14}, before being extended to include HGTs~\cite{sysbio/SzollosiRBTD13,sysbio/SjostrandTDASL14}. 

Most methods that reconstruct a gene tree, conditional to a species tree, rely on the exploration of the space of possible evolutionary histories. 
It is then important to develop conceptual tools that can describe this combinatorial space and further enable its efficient exploration. This naturally raises the questions to compute the size of the space of evolutionary histories for a  given gene family and a given species tree, and to be able to sample such histories.  Both questions are naturally related, as precise counting results often translate into efficient sampling algorithms~\cite{wilf1977unified,flajolet1994calculus}. The former (counting) question has been studied by Rosenberg \textit{et al.}~in the case of the multispecies coalescent model~\cite{jcb/Rosenberg07,jcb/DisantoR15,tcbb/DisantoR16,jcb/DisantoR17,bmb/DisantoR17,jmb/DisantoR18}. However similar questions have not been explored as thoroughly for evolutionary models including gene duplication, gene loss and HGT. In this framework, dynamic programming equations aimed at computing a parsimonious reconciled gene tree can be turned into a specification of the corresponding search space~\cite{tcs/GoreckiT06,jmb/RanwezSDB16}. This then leads to efficient algorithms for counting or sampling parsimonious reconciliations~\cite{cmb/DoyonCH09,cmb/BansalAK13} or sampling reconciled gene trees under the Boltzmann probability distribution~\cite{bioinformatics/JacoxCSPS16}. However, to the best of our knowledge, such questions have not been considered in the case where a gene tree is not specified at first, i.e.~we are only given a species tree and gene family. 

This paper provides analytic and algorithmic answers to those questions. We show that, for a given species tree, whether ranked or unranked, the space of all possible evolutionary histories of a fixed size in the \DLT-model can be described using a formal grammar. This allows us to compute, in polynomial time and space, for given species tree and gene family size, the number of evolutionary histories of this size conditional to the given species tree, as well as to sample among these histories under the uniform probability. Using these algorithms, we can provide estimates of the exponential growth factor of the number of histories in the \DL-model and \DLT-model. We show that, as expected, including HGT in a model results in an exponential increase of the number of histories. We also notice that  with a ranked species tree, the exponential growth factor of the number of histories in the \DLT-model seems to be almost independent of the chosen species tree.  Finally, using enumerative and analytic combinatorics, we provide exact values for the asymptotic number of histories for two specific species tree: the rooted caterpillar tree and the rooted complete binary tree.

%% file: model.tex
\section{Model: gene families evolutionary histories}
\label{sec:model}

In this section, we introduce the combinatorial objects modeling the evolution of a gene family within a given species tree, that we call \textit{histories}.

\paragraph{Preliminaries on trees.} For a given rooted tree\footnote{In the present work we consider only rooted trees.} $\T$, we say it is \emph{uniquely labeled} if every node has a label, and no two nodes have the same label. For a node $x$ in $\T$, we denote by $\T_x$ the subtree of $\T$ rooted at $x$. In this work, we consider only \emph{binary} and \emph{unary-binary trees}: in a binary tree, every internal node has exactly two children, while in a unary-binary tree, an internal node can have either one child or two children.  If a uniquely labeled tree $\T$ is unordered we take advantage of the nodes labeling to see it as an ordered tree, with the two children of an internal node $x$ being ordered from left to right in increasing order of their labels; so from now on all trees we consider are ordered.  If an internal node $x$ of a tree $\T$ is binary, we denote by $x_\ell$ the left child of $x$ and by $x_r$ its right child; if $x$ is unary, i.e. has a single child, we denote it by $x_c$.  We denote by $r(\T)$ the root of $\T$. For a  node $x$ of $\T$, we denote by $p(x)$ its parent in $\T$. The \emph{size} of a tree $\T$ is the number of its leaves.

A rooted tree describes a partial order on the set of its nodes, and two nodes are said to be \textit{comparable} if one is an ancestor of the other one and \textit{incomparable} otherwise. For a node $u$, we denote by $\overline{C}(u)$ the set of nodes that are incomparable with $u$.

\paragraph{Ranked trees.}
A \textit{ranking} of a tree $T$ of size $n$ is a mapping $\pi$ from the nodes of $T$ to $\{1,\dots,n\}$ such that (1) $\pi(x)=n$ if $x$ is a leaf, (2) $\pi(x)\neq\pi(y)$ if $x$ and $y$ are internal nodes, and (3) $\pi(x) < \pi(y)$ if $x$ is an ancestor of $y$.  A tree augmented with a ranking is called a \textit{ranked tree}; in our context it models the evolution of a set of species, the ranking providing the relative order of speciation events, under the assumption that no two speciations can occur at the same time. 

Given a binary tree $\T$ and a ranking $\pi$, we define an unranked unary-binary tree $\T_\pi$ that encodes the ranking information as follows: for each internal node $u$, considered iteratively in increasing ranking order, and for every edge $(p(v),v)$ such that  $\pi(p(v))<\pi(u)<\pi(v)$, we subdivide the edge $(p(v),v)$ into two edges $(p(v),v_u)$ and $(v_u,v)$, so adding a unary node $v_u$ on this edge. We denote by $t(u)$ the set of all unary nodes created in this way and we call this set of nodes together with $u$ a \textit{time slice}. 
Additionally, we also define the set of all leaves as a time slice (see Figure~\ref{fig:ranked_tree}).
Note that in this way we create $n$ different time slices which correspond to the $n$ different values of the ranking.
We modify the notion of incomparability for such unary-binary trees as follows: for a node $u$, $\overline{C}(u)=t(u)\setminus\{u\}$.

\begin{figure}[htbp]
    \begin{center}
        \includegraphics[width=.9\textwidth]{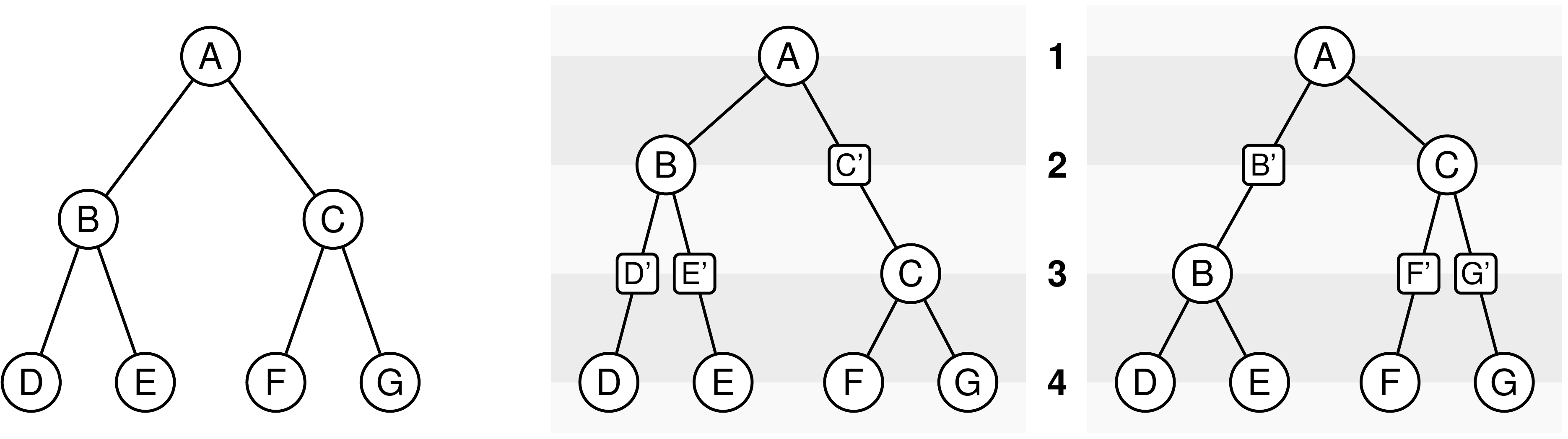}
    \end{center}
    \caption{An example of a ranked tree with time slices. (Left) The complete binary tree $\T$ of size $4$. (Center) The unary-binary tree $\T_\pi$ for the ranking $\pi$ defined by $\pi({\sf A})=1$, $\pi({\sf B})=2$, $\pi({\sf C})=3$ and $\pi({\sf D})=\pi({\sf E})=\pi({\sf F})=\pi({\sf G})=4$; the time slices in $\T_\pi$ are the following sets of nodes: $\{{\sf A}\}$, $\{{\sf B},{\sf C}'\}$, $\{{\sf C},{\sf D}',{\sf E}'\}$,$\{{\sf D},{\sf E},{\sf F},{\sf G}\}$; (Right) Alternative unary-binary tree $\T_{\pi'}$, induced by exchanging the rankings of {\sf B} and {\sf C}. }
    \label{fig:ranked_tree}
\end{figure}

\paragraph{Gene Families Evolutionary Histories.}
The  objects we study in this work model the evolution of a gene family within a species tree. A {species tree}, which will be denoted by $\Sp$ from now on,  is a uniquely labeled rooted binary tree that represents the evolution of a set of species through speciation events; $\Sp$ can be either unranked or ranked. A gene family evolves within $\Sp$ from a single ancestral gene, present in the species $r(\Sp)$, through four possible kinds of \textit{evolutionary events}: 
\begin{itemize}
\item \textit{Speciation} $\Spec$: a gene $x$ present in species $u$ has two descendant genes $x_\ell$ present in species $u_\ell$ and $x_r$ present in species $u_r$.
\item \textit{Duplication} $\Dup$: a gene $x$ present in species $u$ is duplicated, with a new copy $x_d$ of $x$ appearing in species $u$; $x$ is said to be the \textit{original gene} while $x_d$ is the \textit{novel gene}.
\item \textit{Loss} $\Loss$: a gene $x$ present in species $u$ has exactly one descendant either in  $x_\ell$ or in $x_r$, implying that after a speciation at species $u$, exactly one of the two resulting genes is lost along the branch toward either $u_\ell$ or $u_r$.
\item \textit{Horizontal Gene Transfer $\Trans$ (HGT)}: this is similar to a duplication but the novel copy, denoted $x_t$ here, appears in a species $v$ different from $u$ and incomparable with $u$, called the \textit{receiver} of the HGT, while $u$ is called the \textit{donor} of the HGT. If $\Sp$ is ranked, with ranking $\pi$, the receiver species $v$ is required to exist at the same time as $u$, i.e. to satisfy two ranking constraints,  $\pi(p(v))<\pi(u)<\pi(v)$.
\end{itemize}

\begin{definition}
    \label{def:history}
    An evolutionary history for a gene family within a species tree $\Sp$ is a unary-binary ordered rooted tree $\T$ together with two mappings $s:\ V(\T) \rightarrow V(\Sp)$ and $e:\ V(\T) \rightarrow \{\Spec, \Dup, \Loss, \Trans, \Extant\}$ satisfying the following constraints:
    \begin{itemize}
        \item if $x$ is a leaf, $e(x) \in \{\Extant,\Loss\}$;
        \item if $x$ is internal and binary, $e(x) \in \{\Spec, \Dup, \Trans\}$;
        \item if $x$ is internal and unary then $e(x)=\Spec$\footnote{Note that technically the event associated to a unary node in the species tree is not speciation in the biological meaning, but we chose to label it as such for expository reasons.}; 
        \item if $e(x) = \Spec$ and $s(x)=u$ is binary then $s(x_\ell)=u_\ell$ and $s(x_r)=u_r$;
        \item if $e(x) = \Spec$ and $s(x)=u$ is unary then $s(x_c)=u_c$;
        \item if $e(x) = \Dup$ then $s(x_\ell)=s(x_r)=s(x)$;
        \item if $e(x) = \Trans$ then $s(x_\ell)=s(x)$ and $s(x_r) \in \overline{C}(s(x))$.
    \end{itemize}
    The size of a history is the number of leaves $x$ such that $e(x)=\Extant$.
\end{definition}
Intuitively, this definition states that a history is represented by a tree  where each node corresponds to a gene present in a species, either extant or ancestral (the mapping $s$), and each ancestral gene either was lost ($e(x)=\Loss$) or evolved toward extant genes through a duplication ($e(x)=\Dup$), an HGT to an incomparable receiver species ($e(x)=\Trans$) or  a speciation ($e(x)=\Spec$), while extant genes belong to extant species; the constraints on the species mapping $s$ ensure that this history can be embedded within $\Sp$ as illustrated in Figure~\ref{fig:histories}.

By convention, for duplications, we consider that the novel copy of a gene $x$ is its right child $x_r$, $x_\ell$ representing the original copy. Histories considered by the \DL-model, which allows both  duplications and losses (resp. duplications, losses and HGTs), are called \DL-histories (resp. \DLT-histories).

\begin{remark}
    \label{rem:hist_vs_rec}
     By modeling the evolution of a  gene family with ordered trees we differ from the classical notion of \textit{reconciliation}, that also models the evolution of a gene family  but considers that when a gene duplication occurs, the original gene and the novel gene are indistinguishable. As a result, the children of a duplication are ordered within a history, whereas they are not in a reconciliation. 
\end{remark}

\begin{remark}
    \label{rem:loss_model}
    Gene losses are modeled as speciation events with one disappearing gene. As a consequence, we can not have a duplication or a HGT that results in one of the resulting two gene copies being lost. This is necessary to avoid creating an infinite number of histories of a given size, due to an arbitrary number of duplications within a species, each followed by a loss, or an arbitrary long sequence of HGT, again each followed by a loss, leading to at most one extant gene. 
\end{remark}

\paragraph{Time Consistency of \DLT-histories.}
Given an unranked species tree $\Sp$, a \DLT-history as defined above is \textit{time inconsistent} if there exists a gene $x$ belonging to a species $u$ such that one of its ancestors belongs to a species $v$ and one of its descendants belongs to a species $v'$ ancestral to $v$. This pattern can be observed due to the fact that, in the definition of a \DLT-history, the choice of the receiver species $v$ of an HGT of gene $x$ belonging to species $u$ is not restricted to the set of species that are also incomparable with all species containing genes that are ancestral to $x$; see Figure~\ref{fig:time_inconsistency} for an illustration.

\begin{figure}[htbp]
    {\centering
    \includegraphics[width=\textwidth]{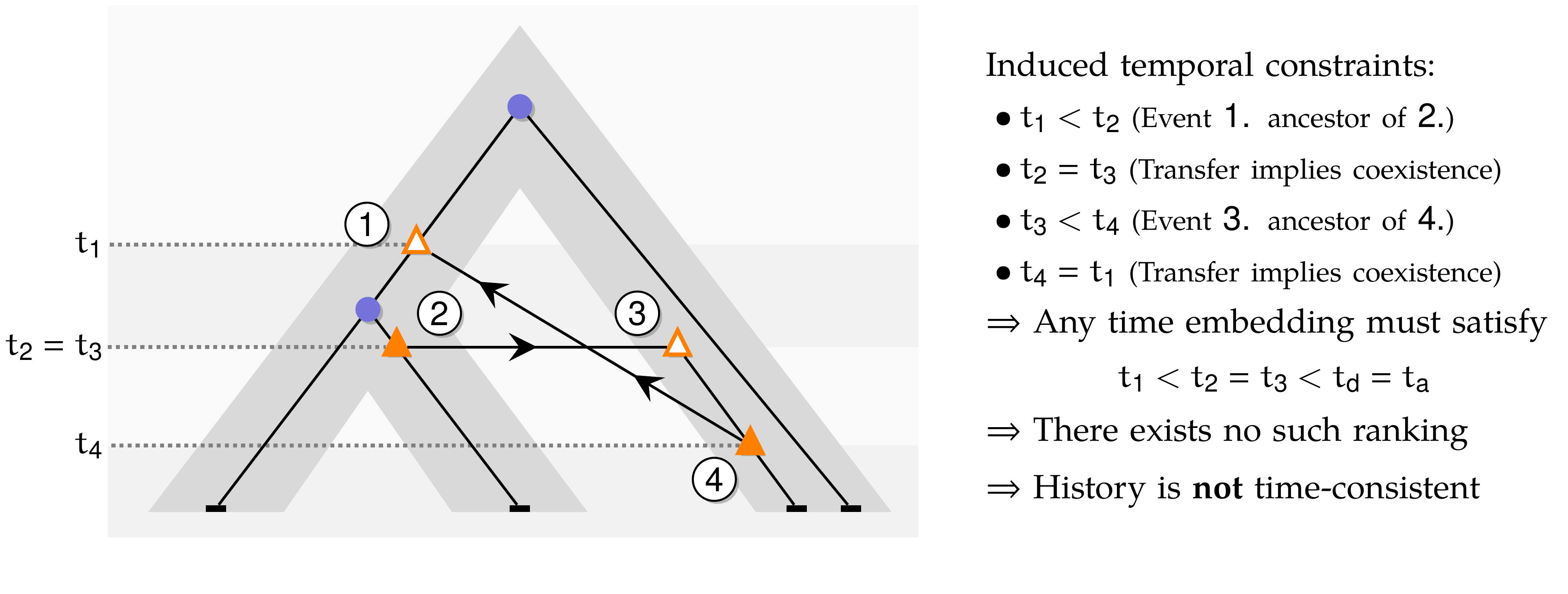}\\}
    \caption{An example of time-inconsistent \DLT-history}
    \label{fig:time_inconsistency}
\end{figure}

The problem of computing gene family evolutionary scenarios 
that are both parsimonious and time-consistent has been shown to be intractable when such scenarios are modeled by reconciliations with an unranked species tree~\cite{tcbb/TofighHL11,cmb/OvadiaFCL11}, while, when the provided species tree $\Sp$ is ranked, the problem becomes tractable (see~\cite{bib/DoyonRDB11} and references therein). Similarly, when $\Sp$ is ranked, we can ensure time-consistency of evolutionary histories, by requiring that the donor and receiver of any HGT  belong to the same time slice in $\Sp_{\pi}$, \textit{i.e.} the receiver of an HGT of a gene belonging to a species $u$ belongs to $\overline{C}(u)=t(u)-\{u\}$.


%% file: methods.tex
\section{Methods}
\label{sec:methods}

Our results (counting and sampling algorithms) are based on the design of formal grammars specifying, for a given species tree $\Sp$, the combinatorial families of \DL-histories and \DLT-histories constrained by $\Sp$. These grammars are then used as templates to design dynamic programming algorithms for counting and sampling (under the uniform distribution) the number of histories of a fixed size. Moreover, these grammars are amenable to techniques of analytic combinatorics that allow us to compute the asymptotic growth constant for the number of histories. We first describe our grammars, then the counting and sampling algorithms, and finally the asymptotic analysis of these grammars.


\subsection{General grammars specifying \DL{}-histories and \DLT-histories}
\label{ssec:DLTgrammar}
\label{ssec:rankedDLTgrammar}

In this section we describe grammars specifying histories evolving within a species tree using the formalism developed in~\cite{comb/FlajoletS09}. We describe grammars for \DLT-histories, for both an unranked and a ranked species tree; these grammars can then be specialized into grammars for \DL-histories by omitting the rules related to HGT.

Let $\Sp$ be a species tree. If $\Sp$ is unranked, it is a binary tree, otherwise, if it comes with a ranking $\pi$, we consider the unary-binary species tree $\Sp_\pi$. So in the statements below, when mentioning a ranked species tree we mean the unary-binary tree $\Sp_{\pi}$ defined by the ranking. 

We denote by $\Hc_u$ the set of \DLT-histories for the tree $\Sp_u$. In the most general setting, following~\cite{comb/FlajoletS09}, these grammars contain both terminal symbols, corresponding to atomic elements of the histories (nodes) and non-terminal symbols,  corresponding to combinatorial operators applied to sets of histories. We use the non-terminal $\Zc_u$ to encode a gene present in extant species $u$; moreover, we use $\Xc_u$ for a gene lost at species $u$, $\Yc_u$ for a duplication at species $u$ and $\Wc_u$ for a HGT with donor species $u$. We consider two combinatorial operators, $\cup$ the disjoint union and $\times$ the Cartesian product.

\begin{theorem}
    \label{thm:uDLT-grammar}
    \label{thm:rDLT-grammar}
    The set $\Hc_{r(\Sp)}$ defined by the grammar below specifies the set of all \DLT-histories for a species tree $\Sp$.
    \begin{align}
    \Hc_u & =  \Sc_u \cup \Dc_u \cup \Tc_u                      & \mbox{if $u$ is internal} \label{eq:DLTgHa}  \\
    \Hc_u & =  \Zc_u \cup \Dc_u \cup \Tc_u                      & \mbox{if $u$ is a leaf}   \label{eq:DLTgHe}\\
    \Sc_u & =  \Hc_{u_\ell} \times \Hc_{u_r} \cup \Hc_{u_\ell} \times \Xc_{u_r} \cup \Xc_{u_\ell} \times \Hc_{u_r}  & \mbox{if $u$ is internal and binary} \label{eq:DLTgS} \\
    \Sc_u & =  \Hc_{u_c}  & \mbox{if $u$ is internal and unary}\label{eq:rDLTgSb}\\
    \Dc_u & =  \Hc_u \times \Hc_u \times \Yc_u  & \label{eq:DLTgD}\\
    \Tc_u & =  \bigcup_{v \in \overline{C}(u)} \Hc_u \times \Hc_v \times \Wc_u  & \label{eq:DLTgT}
    \end{align}
    where $\overline{C}(u)$ is the set of nodes that are incomparable with $u$ in $\Sp$. The set of \DL-histories is specified by the same grammar where rule~(\ref{eq:DLTgT}) is removed and the terms $\Tc_u$ are removed from rules~(\ref{eq:DLTgHa}) and~(\ref{eq:DLTgHe}).
\end{theorem}

\begin{proof}
    The grammar follows the definition of histories, Definition~\ref{def:history}.
    Rule~(\ref{eq:DLTgHa}) simply states that the root (\textit{i.e.} the first evolutionary event of the history) of a \DLT-history within the subtree $\Sp_u$, assuming it is not reduced to a leaf, is either a speciation, a duplication or a transfer of the ancestral gene present in species $u$: non-terminal $\Sc_u$, $\Dc_u$ and $\Tc_u$ represent 
    respectively these three subsets of $\Hc_u$. Rule~(\ref{eq:DLTgHe}) addresses the case where $\Sp_u$ is composed of a single leaf, in which case there can not be a speciation event, but a history reduced to a single gene in species~$u$. 
    
    Rule~(\ref{eq:DLTgS}) describes a speciation event at species $u$. The ancestral gene can either evolve into a gene in each of the two children of $u$ (first term of the union) or into a gene in a single child of $u$ due to a gene loss in the other child of $u$. In the case where $u$ is unary (due to being a node created by the time slicing in a ranked $\Sp$), the ancestral gene evolves into a copy in the unique child $u_c$ of $u$.
    
    Rule~(\ref{eq:DLTgD}) addresses the case of a duplication. It results in two ordered independent histories starting at species $u$: the first one being the history of the original copy of the starting ancestral gene and the second one the history rooted at the novel gene created by the duplication.
    
    Last, Rule~(\ref{eq:DLTgT}) addresses the case of histories starting by a HGT. Generally, a HGT has a structure similar to a duplication  but for the fact that the novel gene appears in a species that is incomparable with $u$.
    
    These various rules cover all cases for describing the possible first event of a history and are mutually exclusive, thus providing a complete recursive specification of $\DLT$-histories for a given species tree $\Sp$. It follows immediately that removing the rule and non-terminals associated to HGT gives a grammar specifying \DL-histories for $\Sp$.
    \hfill $\qed$
\end{proof}



\begin{remark}
    \label{rem:speciesAgnostic}
    The above grammar can be greatly simplified if one is interested only in the number of histories of a given size, as opposed to the specific species where gene duplication, gene loss and HGT events occur and the precise gene content of extant species.
    In this case, one simply identifies all non-terminals $\Zc_u$ (resp. $\Xc_u$, $\Yc_u$, $\Wc_u$) to a single variable $\Zc$ (resp. $\Xc$, $\Yc$, $\Wc$). From now, we follow this approach.
\end{remark}

\subsection{Counting and sampling algorithms}
\label{ssec:algorithms}

The grammar defined above can naturally be turned into a dynamic programming algorithm computing the number of histories of a given size. This algorithm computes tables $H,D,S,T$ where, for a given node $u$ of $\Sp$ and a given history size $n$, $H[u,n]$ (respectively, $D[u,n]$, $S[u,n]$, $T[u,n]$) is the number of $\DLT$-histories of size $n$ evolving within $\Sp_u$ (respectively, starting with a duplication, a speciation, and an HGT). We illustrate this in the case of $\DLT$-histories with an unranked species tree $\Sp$.

\begin{align}
H[u,n] & = S[u,n] + D[u,n] + T[u,n]              & \mbox{if $u$ is internal} \label{eq:cDLTgHa} \\
H[u,n] & = \mathbbm{1}_{n=1} + D[u,n] + T[u,n]                   & \mbox{if $u$ is a leaf}    \label{eq:cDLTgHe} \\
S[u,n] & =  \sum_{m=1}^{n-1}\left(H[{u_\ell},m]H[{u_r},n-m]\right) + H[{u_\ell},n] + H[{u_r},n]  \label{eq:cDLTgS} & \mbox{if $u$ is internal}\\
D[u,n] & =  \sum_{m=1}^{n-1}\left(H[u,m]H[u,n-m]\right)  & \label{eq:cDLTgD}\\
T[u,n] & =  \sum_{m=1}^{n-1}\left(\sum_{v \in \overline{C}(u)} H[u,m]H[v,n-m]\right)  & \label{eq:cDLTgT}
\end{align}

A random generation algorithm can then be adapted from the counting recurrences, resulting in an instance of the so-called recursive method~\cite{wilf1977unified}. Right-hand sides of the counting equation are split into sums of multiplicative terms. Starting from the initial state $H[r(\Sp),n]$, the algorithm randomly chooses a term from the right-hand side of the current state, with probability proportional to its contribution to the counting. 
When the selected term is a multiplication of two terms, the length $n$ needs to be distributed across the two terms, and a pair of lengths $(m, n-m)$, is chosen with probability proportional to the associated count. For the sake of performances, the various alternatives can be explored in Boustrophedon order,  ensuring an overall $\mathcal{O}(n\log(n))$ worst-case complexity~\cite{flajolet1994calculus}.
Recursive calls are then performed over the states associated with the chosen term, until a leaf is chosen (term $\mathbbm{1}$). This leads to the following result.

\newcommand{\CompTime}{\Phi(n,k)}
\newcommand{\CompSpace}{\Psi(n,k)}
\newcommand{\CompSample}{\Upsilon(n,k)}
\begin{table}
    $$
    \begin{array}{c@{\hspace{1em}}c@{\hspace{1em}}c}\toprule
      \text{Counting Time }\CompTime & \DL & \DLT\\ \midrule 
        \text{Unranked} & k\,n^2 & k^2\,n^2\\
        \text{Ranked} & k^2\,n^2 & k^3\,n^2\\ \bottomrule
    \end{array}\quad\quad
    \begin{array}{c@{\hspace{1em}}c@{\hspace{1em}}c}\toprule
      \text{Counting Space }\CompSpace & \DL & \DLT\\ \midrule 
        \text{Unranked} & k\,n^2 & k^2\,n^2\\
        \text{Ranked} & k^2\,n^2 & k^3\,n^2\\ \bottomrule
    \end{array}
    $$ $$
    \begin{array}{c@{\hspace{1em}}c@{\hspace{1em}}c}\toprule
      \text{Generation Time }\CompSpace & \DL & \DLT\\ \midrule 
        \text{Unranked} & n\log n & k\,n\log n\\
        \text{Ranked} & n\log n & k\,n\log n\\ \bottomrule
    \end{array}
    $$
    \caption{Leading terms for the time ($\CompTime$) and space ($\CompTime$) complexities incurred by the evaluation of the counting recurrences for histories consisting of $n$ genes in a species tree of size $k$.}
    \label{tab:complexities}
\end{table}

\begin{theorem}
    \label{thm:counting-sampling}
    The number of histories of size $n$ constrained by a species tree of size $k$ can be computed in polynomial time $\LandauO(\CompTime)$ and space $\LandauO(\CompSpace)$, where $\CompTime$ and $\CompSpace$ both depend on the model ($\DL$ or $\DLT$) and the ranked/unranked nature of the species tree, as summarized in Table~\ref{tab:complexities}.

    The uniform random generation of $h$ histories of size $n$ can be performed in time $\LandauO(\CompTime + h\cdot \CompSample)$.
\end{theorem}

\subsection{Asymptotic number of histories in the \DL{}-model}
\label{ssec:asymptotics}


The grammar given in Theorem~\ref{thm:uDLT-grammar} defines a combinatorial specification of the set of histories for a given species tree in a given evolutionary model. 
In this section, we derive the asymptotic number of histories in the \DL{}-model and use it later on two specific species trees: the caterpillar and complete binary trees.
The following theorem is the main result of this section and  describes their asymptotic growth for $n$ tending to infinity. 

\begin{theorem}
    \label{thm:DLasymgen}
    For any given species tree $\Sp$, the number of histories in the unranked $\DL$-model given by Equations \eqref{eq:DLTgHa}-\eqref{eq:DLTgD} is, for large $n$, equal to
    \begin{align}
        \label{eq:catasymgen}
	    \gamma_\Sp \frac{\rho_\Sp^{-n}}{n^{3/2}}\left(1 + \LandauO\left(\frac{1}{n}\right)\right), 
    \end{align}
    for explicitly computable constants $\gamma_\Sp>0$ and $\rho_\Sp \in (0, 1/4]$.
\end{theorem}

In the remainder of this section we prove this theorem. 
The grammars are amenable to enumerative and analytic combinatorics techniques. 
We follow the general approach presented in Flajolet and Sedgewick~\cite{comb/FlajoletS09} and Drmota~\cite{comb/Drmota97}. It consists mainly in translating the combinatorial specification of a combinatorial family into equations defining its counting generating function. Then, its analytic properties lead to precise asymptotic formulas for its coefficients. We provide an overview of this approach in Example~\ref{ex:catalan}.

\begin{example}
    \label{ex:catalan}
    Consider the class of rooted binary trees $\Bc$. 
    Such a tree is either a leaf, or it consists of a root with two children which are also each roots of binary trees. 
    Let us mark each leaf with the variable $\Zc$. 
    Then, the grammar is given by
    $$\Bc = \Zc \cup \Bc^2.$$
    Let $b_n$ be the number of binary trees with $n$ leaves and let $B(z) = \sum_{n \geq 1} b_n z^n$ be the counting generating function of binary trees. 
    The symbolic method~\cite[Part A]{comb/FlajoletS09} translates this grammar directly into an equation for the generating function:
    \begin{align}
    \label{eq:algbinary}
	    B(z) = z + B(z)^2.
    \end{align}
    Its generating function is thus given by
    $B(z) = \frac{1-\sqrt{1-4z}}{2}.$
    
    The general method of singularity analysis from analytic combinatorics~\cite[Chapter~VI]{comb/FlajoletS09} allows us to directly get the asymptotics of the coefficients. 
    First, by the Cauchy–Hadamard theorem, the asymptotic growth is directly connected with the dominant singularities (and the radius of convergence) of the counting generating function. 
    Here, the generating function $B(z)$ becomes singular at $z=1/4$, which is also the unique singular point. Hence, the coefficients $b_n$ grow like $4^n$.
    Second, using transfer theorems of analytic combinatorics~\cite[Theorem~VI.1 and Theorem~VI.3]{comb/FlajoletS09} we also get the subexponential terms and recover the well-known result for Catalan numbers $b_{n+1} = \frac{1}{n+1}\binom{2n}{n}$ (see~\OEIS{A000108}~\cite{comb/Sloane}):
    \begin{align*}
        b_n &= \frac{4^{n-1}}{\sqrt{\pi n^3}} \left(1 + \LandauO\left(\frac{1}{n}\right) \right),
    \end{align*}
    for $n \to \infty$. 
    \exampleend
\end{example}

We will now describe this approach applied to the grammar specifying the $\DL$-histories with an unranked species tree $\Sp$. 
Let $h_{u,n}$ be the number of $\DL$-histories of $\Sp_u$ consisting of $n$ genes represented in the generating function by the formal variable~$z$. 
We define the counting generating functions
\begin{align*}
	H_u(z) = \sum_{n \geq 0} h_{u,n} z^n.
\end{align*}
The coefficients $h_{u,n}$ represent the number of histories of size $n$ associated with the species tree $\Sp_u$ independent on the number of losses or duplications. 
These generating functions (one per species $u$ of $\Sp$) are strongly related to the generating function of binary trees $B(z)$ 
introduced in Example~\ref{ex:catalan}. 
\begin{lemma}
\label{lem:HuBgen}
For a given species tree $\Sp$ the counting generating function $H_{r(\Sp)}(z)$ for histories in the unranked $\DL$-model is defined by the system of functional equations
\begin{align}
\label{eq:Hgen}
\begin{aligned}
	H_u(z) &= B \left( H_{u_{\ell}}(z)H_{u_r}(z) + H_{u_{\ell}}(z) + H_{u_r}(z) \right)&& \text{ if $u$ is internal,}\\
    H_u(z) &=  B \left(z\right) && \text{ if $u$ is a leaf,}
\end{aligned}
\end{align}
over all nodes $u$ of $\Sp$, where
$$B(z) = \frac{1-\sqrt{1-4z}}{2}.$$
\end{lemma}

\begin{proof}
    The symbolic method~\cite[Part A]{comb/FlajoletS09} translates
    the unranked $\DL$-grammar of Equations \eqref{eq:DLTgHa}-\eqref{eq:DLTgD} directly into a system of equations for the generating functions. We get
    \begin{align}
    \label{eq:funceqgen}
    \begin{aligned}
    	H_{u}(z) &= H_u(z)^2 + H_{u_{\ell}}(z) H_{u_r}(z) + H_{u_{\ell}}(z)  + H_{u_r}(z) && \text{ if $u$ is internal,}\\
        H_u(z) &= H_u(z)^2 + z && \text{ if $u$ is a leaf.}
    \end{aligned}
    \end{align}
    Comparing these equations with the one for binary trees from Equation~\eqref{eq:algbinary} the claim follows.
    \hfill $\qed$
\end{proof}


The advantage of a generating function approach is that we are able to identify the subexponential growth as $n^{-3/2}$, and that we are able to explicitly compute exponential growth $\rho_\Sp^{-1}$ and the constant $\gamma_\Sp$ for a fixed species tree $\Sp$.
We will compute the involved constants explicitly for the caterpillar tree in Section~\ref{ssec:DLCaterpillar} and for the complete binary tree in Section~\ref{ssec:DLcomplete}.

By basic principles of analytic combinatorics, the asymptotic growth of a counting sequence is directly related to the radius of convergence of the corresponding generating function.
In particular, its dominant singularity (i.e.~the one closest to the origin) defines its asymptotic growth. 
By the construction in terms of nested radicals, the generating function $H_u(z)$ is singular if and only if at least one of its radicals becomes zero. 
Therefore, we make the structure of nested radicals visible.
Writing the explicit form of the outermost $B(z)$ in~\eqref{eq:Hgen} gives
\begin{align}
    \label{eq:HusqrtRu}
    H_u(z) = \frac{1- \sqrt{R_u(u)}}{2}.
\end{align}
Then, the radicands satisfy the following recurrence
\begin{align}
	\label{eq:comPk}
	\begin{aligned}
    	R_{u}(z) &= -4 + 3 \sqrt{R_{u_{\ell}}(z)} + 3 \sqrt{R_{u_r}(z)} - \sqrt{R_{u_{\ell}}(z) R_{u_r}(z)} && \text{ if $u$ is internal,}\\
    	R_u(z) &= 1-4z && \text{ if $u$ is a leaf.}
    \end{aligned}
\end{align}

The recurrence can be used to determine the nature of the radii of convergence.
For a node $u$ we define $\rho_u$ as the radius of convergence of $H_u(z)$.

\begin{lemma}
	\label{lem:uniquesinggen}
	Let $u$ be the parent of $v$ in $\Sp$. 
	Then, $\rho_u < \rho_v$ and $\rho_u \in (0,1/4]$ with $\rho_u=1/4$ if $u$ is a leaf. 
	Furthermore, $R_u(z)$ is the only radicand that vanishes at $z=\rho_u$ and $\rho_u$ is a simple root.
\end{lemma}

\begin{proof}
    By combinatorial construction $H_u(z)$ is built of nested radicals and does not include any poles. 
    Therefore, its dominant singularity must be at a point where (at least) one of its radicands vanishes.
    
    We continue by induction on the depth of the subtree with root $u$ given by $\Sp_u$.
    The depth is the longest path from the root to any leaf.
    As a first step, we prove that $R_{u}(0)=1$ and that $\rho_{u} \leq 1/4$. 
    For a leaf $u$ it is clear from Relation~\eqref{eq:comPk} that $R_u(0) = 1$ and that $\rho_u=1/4$. 
    
    Next, let $v$ and $w$ be the children of $u$ such that $\rho_v \leq \rho_w$.
    By the induction hypothesis we directly get 
    $$
    	R_{u}(0) = -4 + 3 \sqrt{R_{v}(0)} + 3 \sqrt{R_{w}(0)} - \sqrt{R_{v}(0) R_{w}(0)} = 1.
    $$
    In order to continue, note that $R_u(z)$ is monotonically decreasing on $[0,+\infty]$, because from the decomposition in~\eqref{eq:HusqrtRu} and \eqref{eq:funceqgen} we see that 
    \begin{align}
        \label{eq:Ruseries}
        R_u(z) = 1 - \sum_{n \geq 1}{a_n} z^n,
    \end{align}
    for certain non-negative numbers $a_n$. 
    
    By the induction hypothesis and Relation~\eqref{eq:comPk}, $R_u(z)$ is a continuous function on $(0,\rho_{v})$. 
    Hence, we get
    \begin{align*}
    	R_{u}(\rho_{v}) = -4 + 3 \sqrt{R_w(\rho_v)} < 0.
    \end{align*}
    Thus, on the one hand, by the intermediate value theorem $R_u(z)$ must have at least one zero in the interval $(0,\rho_{v})$. 
    On the other hand, as $R_u(z)$ is monotonically decreasing it has at most one zero in $(0,\rho_{v})$. 
    Hence, this zero is equal to $\rho_u$.
    
    Finally, the above reasoning implies that among the nested radicals of $H_u(z)$ the outermost one is the first one that vanishes, and no other radical vanishes at the same time. 
    Thus, $\rho_u$ is the radius of convergence of $H_u(z)$.
    Moreover, by~\eqref{eq:Ruseries} we see that the derivative $R_u'(z)$ has non-positive coefficients. 
    Hence, $\rho_u$ is a simple root.
    \hfill $\qed$
\end{proof}


Let us shortly digress and discuss in a more general context how to numerically compute the exponential growth for the coefficients of the generating function with the fastest exponential growth that is defined by a system of functional equations involving generating functions $\Bc_1,\dots,\Bc_k$ of the form
\begin{align*}
    B_{i} = \Phi_i\left(z, B_{1}, \ldots, B_{k} \right),
\end{align*}
where the $\Phi_i$ are polynomials with non-negative integer coefficients in $k+1$ variables. 
Note that the grammar given in Theorem~\ref{thm:uDLT-grammar} is of this shape.
In order to decide which of the $B_i$'s has this specific exponential growth, further information on the problem, like in our case given by Lemma~\ref{lem:uniquesinggen}, is needed.
By Banach's fixed point theorem, these equations admit a unique solution vector $(B_{1},\ldots,B_{k}) \in (\CC[[z]])^k$ with respect to the formal topology~\cite[Section~A.5]{comb/FlajoletS09}.
Furthermore, each $B_{i}(z)$ has non-negative coefficients in its expansion around $0$ (which is already clear from the combinatorial nature of the problem).
Then, the multivariate version of the implicit function theorem implies that each of them has a non-zero radius of convergence which we call $\rho_i$. 
By Pringsheim's Theorem~\cite[Theorem~IV.6]{comb/FlajoletS09}, $\rho_i \in [0,+\infty]$ is a singularity of $B_{i}(z)$.
Moreover, as $B_{i}(z)$ is an ordinary generating function of an infinite combinatorial class, we must have $\rho_i \in [0,1]$.
Finally, in order to compute the radius of convergence, we find the minimal point $z \in [0,1]$ where the implicit function theorem fails. 
To be more precise, we numerically compute solutions $\rho \in [0,1]$ and $b_1,\ldots,b_k \in [0,+\infty)$ of the following system
    \begin{align*}
    \left\{
    \begin{array}{rl}
        b_1 &= \Phi_1(\rho,b_1,\ldots,b_k)\\
            &\quad\vdots\\
        b_k &= \Phi_k(\rho,b_1,\ldots,b_k)\\
        0   &= \det\left(\delta_{i,j} - \frac{\partial}{\partial b_j} \Phi_{i}(\rho,b_1,\ldots,b_k)\right),
    \end{array}
    \right.
    \end{align*}
    where $\delta_{i,j}$ is the Kronecker symbol: $\delta_{i,i}=1$, and $\delta_{i,j}=0$ for $i \neq j$.

\begin{remark}
    \label{rem:DLTsquareroot}
    The unranked $\DL$-grammars lead to the following specific shape
    \begin{align*}
    \left\{
    \begin{array}{rl}
        B_1 &= \Phi_1(z,B_1)\\
        B_2 &= \Phi_2(z,B_1,B_2)\\
            &\quad\vdots\\
        B_k &= \Phi_k(z,B_1,\ldots,B_k)\\
    \end{array}
    \right.
    \end{align*}
    Hence, we get
    $
        \det\left(\delta_{i,j} - \frac{\partial}{\partial b_j} \Phi_{i}(\rho,b_1,\ldots,b_k)\right) = \prod_{i=1}^{k} (1-2b_i).
    $
    We actually know by Lemma~\ref{lem:uniquesinggen} that the outermost square-root vanishes, which gives $b_k = B_k(\rho)=1/2$.
    Additionally, we can also directly deduce from this system that $\rho_{k} \leq \rho_{k-1}$.
    
    In the unranked $\DLT$-model the system looks like 
    \begin{align*}
    \left\{
    \begin{array}{rl}
        B_1 &= \Phi_1(z,B_1,B_2,\ldots,B_{k-1})\\
            &\quad\vdots\\
        B_{k-1} &= \Phi_{k-1}(z,B_1,\ldots,B_{k-1})\\
        B_k &= \Phi_k(z,B_1,\ldots,B_k)\\
    \end{array}
    \right.
    \end{align*}
    where the last equation is the only one involving $B_k$, as the root can not be a receiver of an HGT.
    Note that the subsystem of the first $k-1$ equations is strongly connected and but still not satisfies the $a$-properness condition (i.e.~it is no contraction in the formal topology) of the Drmota--Lalley--Woods Theorem~\cite[Theorem~VII.6]{comb/FlajoletS09} which would directly imply a square root singularity. 
    Thus, we conjecture that the dominant singularity still comes solely from the outermost square root of $B_k$ implying $b_k=1/2$. 
    
    In the ranked  $\DLT$-model we are dealing with blocks of strongly connected components that correspond to the time slices. 
    Note that the root is contained in a singleton time slice. 
    Experiments suggest the same behavior as in the previous cases.
    
    However, one thing is for sure in all models: we always have $\rho_{r(\Sp)} \leq \rho_u$ for all other subtrees with root $u$ of the species tree. 
    Hence, there will be always a dominant minimal singularity in $[0,1]$ that can be (numerically) computed.
    Note however, that the determinant computation soon becomes extremely heavy.
\end{remark}

After determining the radius of convergence, we must determine the number of singularities on it. 
As shown in the case of $\lambda$-terms in~\cite[Lemma~8]{comb/BodiniGGG18} there can only be one dominant singularity $\rho_u$. 
Let us quickly repeat this argument here.
Assume that there exists a root $z_0 = \rho_u e^{i \theta}$ of the same modules. 
Substituting this value into $R_u(z)$ from~\eqref{eq:Ruseries} gives
\begin{align*}
    1 &= \sum_{n\geq 1} a_n \rho_u^n = \big| \sum_{n \geq 1} a_n z_0^n \big|,
\end{align*}
which can only hold if $e^{i n \theta} = 1$ whenever $a_n \neq 0$.
Now, due to $a_1 \neq 0$ we have $z_0=\rho_u$.
Hence, $\rho_u$ is the unique dominant real singularity of $H_u(z)$.

Combining the previous results, we have shown for a family of constants $\gamma_{u,i}$ the following local singular expansion
\begin{align*}
	H_u(z) = \frac{1}{2} - \sum_{i \geq 0} \gamma_{u,i} \left(1-z/\rho_u\right)^{i + 1/2}.
\end{align*}
The fact that $R_u(z)$ has a simple root at $z=\rho_u$ shows that $\gamma_{u,0} > 0$. 
Then, by transfer theorems of analytic combinatorics~\cite[Theorem~VI.1 and Theorem~VI.3]{comb/FlajoletS09}, we get the claimed asymptotic expansion of Equation~\eqref{eq:catasymgen}, where $\gamma_T = \frac{\gamma_{u,0}}{2\sqrt{\pi}} > 0$ and this ends the proof of Theorem~\ref{thm:DLasymgen}. 

\begin{remark}
    There are several possible extensions of the previous approach.
    First of all, it is straightforward to extend it to the ranked $\DL$-model. 
    In that case one only needs to incorporate unary nodes arising from the time slices.
    Second, an extension to the $\DLT$-model is also possible, yet the computations are more involved as the binary tree structure leading to Lemma~\ref{lem:HuBgen} does not hold anymore. However, it can still be modeled with colored binary trees, where the number of colors depends on the size of the set of incomparable nodes (in the the current time slice).
	Third, it is also possible to consider the distribution of certain parameters, such as the number of gene losses, or the number of gene duplications, see e.g. for related results in lattice paths and trees~\cite{comb/BonaF09,comb/GittenbergerJW18,comb/BanderierW17}. Using multivariate generating functions and marking each such event by an additional variable like in the general grammar of Theorem~\ref{thm:uDLT-grammar}, the above results for the $\DL$-model directly generalize to the respective ones on multivariate generating functions.
	All these generalizations are interesting future research directions.
\end{remark}


The counting and sampling algorithms described above have been implemented in \texttt{Python}, and are available at
\url{https://github.com/cchauve/DLTcount}.

%% file: results.tex
\section{Results}
\label{sec:results}

Over the next two sections, we will apply Theorem~\ref{thm:DLasymgen} to the special cases of the caterpillar and complete species tree in the unranked $\DL$-model, and explicitly determine the  constants involved in the asymptotic expansion. Then, we apply our dynamic programming counting and sampling algorithms to study properties of random evolutionary histories.

\subsection{Asymptotic expansion for extremal species trees in the $\DL$-model}

Our experimental results (Section~\ref{ssec:experiments}) suggest that for a given $k$, the species trees having the largest (resp. smallest) number of \DL-histories are respectively the caterpillar tree and the balanced binary tree (Conjecture~\ref{conj:caterpillar-complete}), defined below. In the present section, our main results are the explicit computation of the asymptotic growth and the leading constant of Theorem~\ref{thm:DLasymgen} for the caterpillar species tree (Propositions~\ref{prop:catdomgrowth} and \ref{prop:catleadconst}) and for the complete binary species tree, the special case of balanced trees when $k$ is a power of $2$ (Propositions~\ref{prop:comdomgrowth} and \ref{prop:comleadconst}, see also Table~\ref{tab:catasygrowth}). 


\begin{figure}[htbp]
    \centering
    \includegraphics[width=0.27\textwidth]{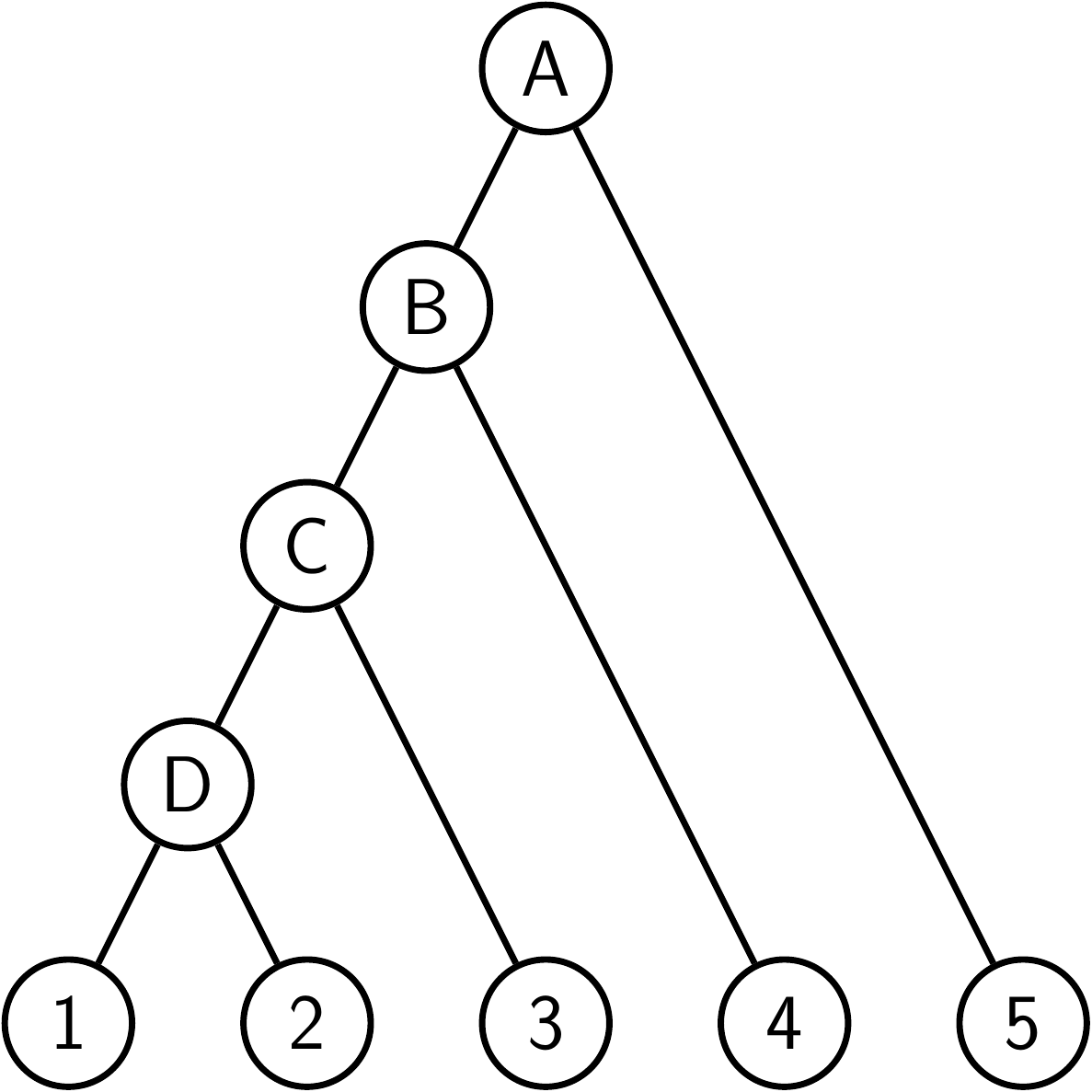}
    \qquad \qquad
    \includegraphics[width=0.3\textwidth]{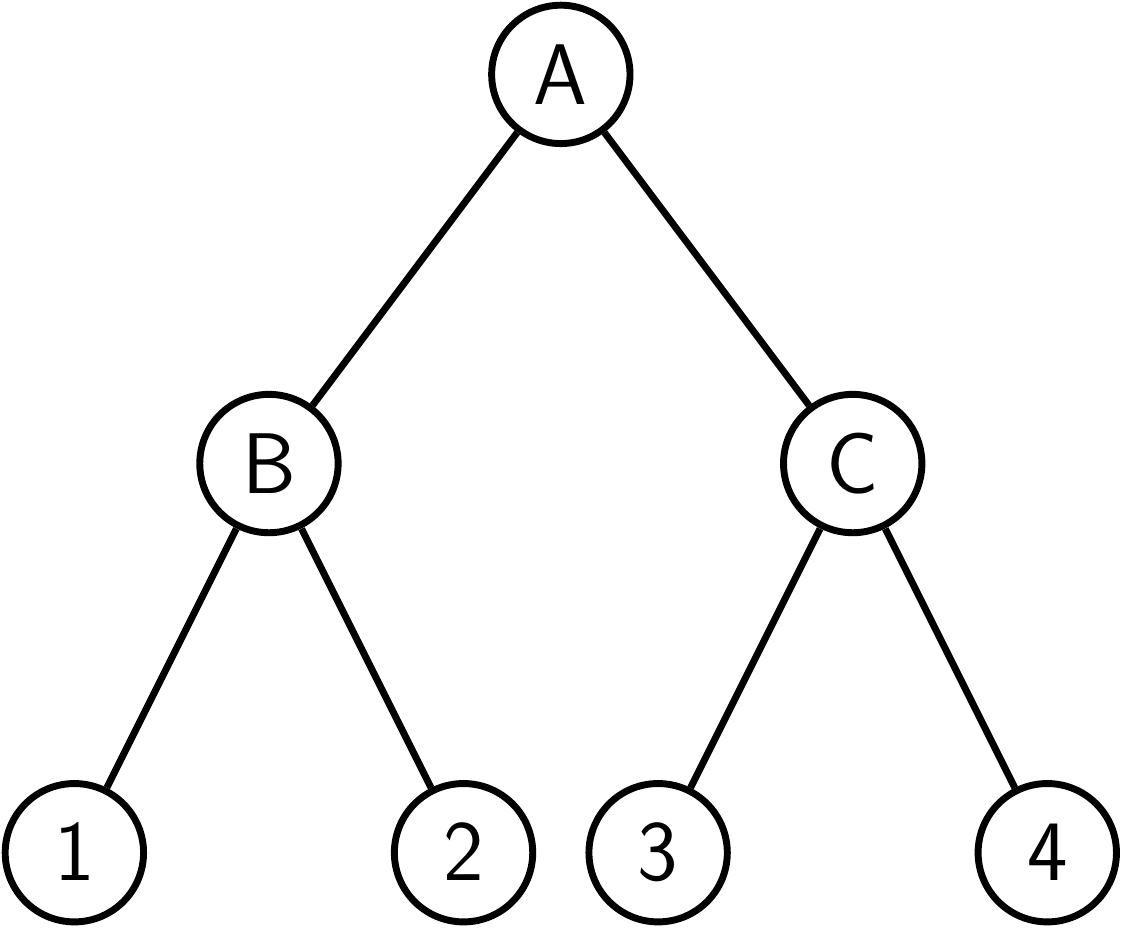}
    \caption{(Left) The caterpillar species tree $\CT_5$. (Right) The complete binary tree $\CB_2$.}
    \label{fig:caterpillar_complete}
\end{figure}

The rooted caterpillar tree $\CT_k$ can be defined as follows: $\CT_1$ is the tree reduced to a single leaf, while $\CT_k$ ($k > 1$) is the tree formed by a left subtree equal to $\CT_{k-1}$ and a right subtree equal to $\CT_1$. 
Observe that every subtree of a caterpillar tree is itself a caterpillar tree, see Figure~\ref{fig:caterpillar_complete}.

The complete binary tree $\CB_h$ with $k=2^h$ leaves can be defined as follows: $\CB_0$ is the tree reduced to a single leaf, while $\CB_h$~($h\geq 1$) is the tree formed by a left and a right subtree both equal to $\CB_{h-1}$. 
Observe again that every subtree is itself a complete binary tree, see Figure~\ref{fig:caterpillar_complete}. The complete binary tree is a special case of the class of \textit{balanced} trees, defined as trees where, for each node, the number of leaves in the left subtree differs from the number of leaves in the right subtree by at most one. Complete binary trees are the only balanced trees in which the number of leaves is a power of two.

We can observe that the number of $\DL$-histories grows much faster for the caterpillar tree than for the complete binary tree. This is actually unsurprising given that the number of $\DL$-histories can be linked to the size of the grammar, which itself depends on the structure of the species tree. More precisely, the size of the grammar depends on the number of unique subtrees of the considered species tree $S$. Each such subtree may be identified by its root $u$ and corresponds to one set of rules~\eqref{eq:DLTgHa}-\eqref{eq:DLTgT}, while subtrees having the same topology lead to isomorphic subgrammars with the same counting generating functions. The caterpillar (resp. complete binary) tree has the largest (resp. smallest) number of unique subtrees within the set of species trees of the same size (when $k$ is a power of $2$ for the complete binary tree), compare also Table~\ref{tab:catasygrowth}.

\begin{table}[ht]
	\centering
    \begin{filecontents*}{caterpillarcomplete.csv}
genes,alpha,lambda,asygrowthlambda,beta,mu,asygrowthmu
1,0.1410,0.2500,4.00,0.1410,0.2500,4.00
2,0.1557,0.1041,9.61,0.1557,0.1041,9.61
3,0.1647,0.0636,15.72,-,-,-
4,0.1742,0.0441,22.69,0.1620,0.0482,20.75
5,0.1835,0.0328,30.53,-,-,-
6,0.1927,0.0255,39.25,-,-,-
7,0.2015,0.0205,48.84,-,-,-
8,0.2101,0.0169,59.31,0.1650,0.0232,43.02
9,0.2184,0.0142,70.65,-,-,-
10,0.2265,0.0121,82.86,-,-,-
11,0.2342,0.0104,95.93,-,-,-
12,0.2418,0.0091,109.85,-,-,-
13,0.2491,0.0080,124.64,-,-,-
14,0.2563,0.0071,140.28,-,-,-
15,0.2632,0.0064,156.77,-,-,-
16,0.2700,0.0057,174.11,0.1664,0.0114,87.56
\end{filecontents*}
\csvreader[tabular=lllll,
    table head=\toprule
    \#Species  & \multicolumn{2}{l}{Caterpillar tree $\CT_k$}& \multicolumn{2}{l}{Complete binary tree $\CB_k$}\\
    $k$ & $\alpha_k$ & Exp.~Growth $\lambda_k^{-1}$ & $\beta_h$ & Exp.~Growth $\mu_h^{-1}$\\\midrule,
    table foot=\bottomrule
    ]{caterpillarcomplete.csv}{%
    1=\one,2=\two,3=\three,4=\four,5=\five,6=\six,7=\seven
    }
    {$\one$ & $\two$ & $\four$ & $\five$ & $\seven$}
    \caption{Leading constants and exponential growth factors for the number of $\DL$-histories consistent with the unranked caterpillar and complete species tree. Their closed forms are given in Propositions~\ref{prop:catdomgrowth}--%
    \ref{prop:comleadconst}.
    }
	\label{tab:catasygrowth}
\end{table}

\input{caterpillar-DL.tex}

\input{complete-DL.tex}

\input{experiments.tex}

%% file: caterpillar-DL.tex
\subsubsection{Counting \DL{}-histories associated with the caterpillar species tree}
\label{ssec:DLCaterpillar}

\newcommand{\Hct}[1]{{H^{\CT{}}_{{#1}}}}
\newcommand{\Dct}[1]{{D^{\CT{}}_{{#1}}}}
\newcommand{\Sct}[1]{{S^{\CT{}}_{{#1}}}}
Denote by $\Hct{k}$ the set of $\DL$-histories over the caterpillar $\CT_k$, then the general grammar of $\DL$-histories, where extant genes are marked by a single terminal $\Zc$, is the following:
\begin{align}
\Hct{k} & =  \Dct{k} + \Sct{k}                        & \mbox{if $k > 1$} \label{eq:DLctHa} \\
\Hct{1} & =  \Zc + \Dct{1}                          &  \label{eq:DLctHe} \\
\Sct{k} & =  \Hct{k-1} \times \Hct{0} + \Hct{k-1}  +  \Hct{0}  \label{eq:DLctS} & \mbox{if $k > 1$}\\
\Dct{k} & =  \Hct{k} \times \Hct{k}    & \label{eq:DLctD}
\end{align}



\newcommand*{\ctcf}{f}
\newcommand*{\ctgf}{F}
Let $\ctcf_{k,n}$ be the number of $\DL$-histories of the caterpillar $\CT_k$ consisting of $n$ genes. The corresponding counting generating function is given by 
$$\ctgf_k(z) = \sum_{n \geq 0} \ctcf_{k,n} z^n,$$
and, by Lemma~\ref{lem:HuBgen}, it is defined by the functional equation
\begin{align*}
	\ctgf_{k}(z) = B \left(\ctgf_{k-1}^2(z) + \ctgf_{k-1}(z) + B(z) \right).
\end{align*}
In Table~\ref{tab:OEISCaterpillar} we computed the first few initial terms for $k=1,...,5$. 
Note that none but the first one was found in the OEIS~\cite{comb/Sloane} before we added them.
\begin{table}
\centering
\begin{tabular}{clc}
	\toprule
	$k$ & Sequence & OEIS \\
	\midrule
	$1$ & 
	$1, 1, 2, 5, 14, 42, 132, 429, 1430, 4862, 16796, 58786, 208012, 742900, \ldots$ &
	\OEISs{A000108}\\
	$2$ & 
	$2,7,34,200,1318,9354,69864,541323,4310950,35066384, \ldots$ &
	\OEISs{A307696}\\
	$3$ & 
	$3,19,159,1565,17022,197928,2413494,30490089,395828145, \ldots$ &
	\OEISs{A307697}\\
	$4$ & 
	$4,39,495,7235,115303,1948791,34379505,626684162, \ldots$ &
	\OEISs{A307698}\\
	$5$ & 
	$5,69,1230,24843,541315,12426996,296546600,7292489761, \ldots$ &
	\OEISs{A307700}\\
	\bottomrule
\end{tabular}
\caption{$\DL$-history counting sequences of the caterpillar species trees $\CT_k$.
}
\label{tab:OEISCaterpillar}
\end{table}
Applying Theorem~\ref{thm:DLasymgen}, the asymptotic expansion of the coefficients for $n \to \infty$ is 
\begin{align}
	\ctcf_{k,n} = \alpha_k \frac{\lambda_k^{-n}}{n^{3/2}}\left(1 + \LandauO\left(\frac{1}{n}\right)\right).
\end{align}
for some constants $\alpha_k>0$ and $\lambda_k>0$ that are made explicit below.

\begin{proposition}
	\label{prop:catdomgrowth}
	Let $a(X) = 3X-4$ and $b(X) = X-3$.
    We define the following sequence of rational functions in $X$
    \begin{align*}
    	\begin{cases}
    	    s_1(X) = 0,\\
        	s_{k}(X) = \frac{a(X) - s_{k-1}(X)^2}{b(X)} & \text{ for } k > 1.
        \end{cases}
    \end{align*}
    Let $X_k$ be the minimal positive real solution of the fixed point equation
    \begin{align*}
    	s_k(X) = X.
    \end{align*}
    Then, the dominant singularity of $\ctgf_{k}(z)$ can be found at $\lambda_k = \frac{1-X_k^2}{4}$.
\end{proposition}

\begin{proof}
    We need to analyze the nested radicals of $\ctgf_k(z)$ in more detail. 
    Therefore, as done in Equation~\eqref{eq:HusqrtRu} for the general case, we define the decomposition
    \begin{align*} 
    	\ctgf_k(z) &= \frac{1 - \sqrt{P_k(z)}}{2}.
    \end{align*}
    Thus, we directly get the specialized version of the reucurrence for the radicands from Equation~\eqref{eq:comPk} by 
    \begin{align}
    	\label{eq:catRk}
    	\begin{cases}
        	P_1(z) = 1-4z, \\
        	P_k(z) = -4 + 3\sqrt{1-4z} + (3-\sqrt{1-4z})\sqrt{P_{k-1}(z)}, &\text{ for } k > 1.
        \end{cases}
    \end{align}

	The dominant singularity $\lambda_k$ is given by the minimal positive root of $P_k(z)$.
	This already proves the case $k=1$.
    We introduce the shorthand $X=\sqrt{1-4z}$ and use it from now on as our new variable. 
    This directly gives
    \begin{align}
        \label{eq:Rkab}
    	P_k(X) = a(X) - b(X) \sqrt{P_{k-1}(X)}.
    \end{align}
    Hence, this equation is zero if and only if
    \begin{align*}
    	\sqrt{P_{k-1}(X)} = \frac{a(X)}{b(X)} =: s_2(X).
    \end{align*}
    For $k=2$ this proves the claim as $\sqrt{P_1(X)}=X$.
    Now we proceed by induction. 
    Squaring this equation and substituting the known expression for $P_{k-1}(X)$ gives
    \begin{align*}
    	\sqrt{P_{k-2}(X)} = \frac{a(X) - s_2(X)^2}{b(X)} =: s_3(X).
    \end{align*}
    Repeating this process proves the claim.
    \hfill $\qed$
\end{proof}


\begin{proposition}
    \label{prop:catleadconst}
    Using the notation of Proposition~\ref{prop:catdomgrowth}, the constant $\alpha_k$ is equal to
    \begin{align*}
        \alpha_k &= 
       \sqrt{ \frac{\lambda_k}{8 \pi X_k}  \sum_{i=2}^{k+1} \sigma_{i,k}(X_k) \left(\frac{3-X_k}{2}\right)^{i-2}  \prod_{j=2}^{i-1} \frac{1}{s_j(X_k)} },
        \\
        \sigma_{i,k}(X) &= 
            \begin{cases}
                3 - s_i(X) & \text{ if } i \leq k, \\
                2 X        & \text{ if } i = k+1. \\
            \end{cases}
    \end{align*}
    In particular, $\alpha_k >0$.
\end{proposition}

\begin{proof}
    We will prove that $P_k(X)$ admits the following extension in a neighborhood of $X_k$:
    \begin{align*}
        P_k(X) &= P_k'(X_k) (X-X_k) + \LandauO((X-X_k)^2),
    \end{align*}
    where the derivative is with respect to $X$.
    Note that this derivative exists, as $P_k(X)$ is analytic on $(0,(1-X_{k-1}^2)/4)$ and we know from Lemma~\ref{lem:uniquesinggen} that $X_{k-1} < X_k$.
    
    Next, recall the shorthand $X=\sqrt{1-4z}$ and that by the chain rule $\partial_z P_k(X) = \partial_X P_k(X) \partial_z X$.
    Then, the transfer theorems of analytic combinatorics~\cite{comb/FlajoletS09} directly show that the $n$-th coefficient of $\ctgf_k(z)$ 
		satisfies the form~\eqref{eq:catasymgen} with $\alpha_k = \sqrt{\lambda_k P_k'(X_k)/(8 \pi X_k)}$. 
    Therefore, it remains to find an expression for $P_k'(X_k)$.
    
    Let us take the derivative of Equation~\eqref{eq:Rkab}. We get
    \begin{align*}
        P_k'(X) &= 3 - \sqrt{P_{k-1}(X)} + \frac{3-X}{2\sqrt{P_{k-1}(X)}}P_{k-1}'(X).
    \end{align*}
    In the proof of Proposition~\ref{prop:catdomgrowth} we have seen that $\sqrt{P_{i}(X_k)} = s_{k-i+1}(X_k)$. Iterating this equation until $P_1'(X)=2X$ shows the claim.
    Finally, the positivity of the constant holds as all terms are positive.
    \hfill $\qed$
\end{proof}

With these formulas it is easy to compute explicit values for the constant $\alpha_k$ and the asymptotic growth factor $\lambda_k^{-1}$. We show the first few values in Table~\ref{tab:catasygrowth}.

%% file: complete-DL.tex
\subsubsection{Counting \DL-histories associated with the complete species tree}
\label{ssec:DLcomplete}

\newcommand{\Hcb}[1]{H^{\CB}_{#1}}
\newcommand{\Dcb}[1]{D^{\CB}_{#1}}
\newcommand{\Scb}[1]{S^{\CB}_{#1}}

Let $\Hcb{h}$ be the set of $\DL$-histories associated with the complete binary tree $\CB_h$. Then, the respective grammar, considering again only terminals $\Zc$ marking extant genes, is the following:
\begin{align}
\Hcb{h} & =  \Dcb{h} + \Scb{h}                         & \mbox{if $h\geq 1$} \label{eq:DLcbHa} \\
\Hcb{0} & =  \Zc + \Dcb{0}                            &  \label{eq:DLcbHe} \\
\Scb{h} & =  \Hcb{h-1} \times \Hcb{h-1} + \Hcb{h-1} + \Hcb{h-1}  \label{eq:DLcbS} & \mbox{if $h\geq 1$}\\
\Dcb{h} & =  \Hcb{h} \times \Hcb{h}   & \label{eq:DLcbD}
\end{align}

\newcommand*{\cbcf}{g}
\newcommand*{\cbgf}{G}

Let $\cbcf_{h,n}$ be the number of histories over the complete binary tree $\CB_h$ consisting of $n$ genes represented by $z$. As before, we analyze the counting generating function which is given by 
\begin{align*}
	\cbgf_h(z) = \sum_{n \geq 0} \cbcf_{h,n} z^n,
\end{align*}
and, by Lemma~\ref{lem:HuBgen}, it is defined by the functional equation
\begin{align*}
	\cbgf_{h}(z) = B \left(\cbgf_{h-1}^2(z) + 2 \cbgf_{h-1}(z) \right).
\end{align*}
As before, we computed the first few initial terms in Table~\ref{tab:OEISComplete}. 
Again, none but the first one was found in the OEIS~\cite{comb/Sloane} before we added them.

\begin{table}
\centering
\begin{tabular}{cclc}
    \toprule
	$h$ & $k$ & Sequence & OEIS \\
	\midrule
	$0$ & $1$ & 
	$1, 1, 2, 5, 14, 42, 132, 429, 1430, 4862, 16796, 58786, 208012, 742900, \ldots$ &
	\OEISs{A000108}\\
	$1$ & $2$ &
	$2,7,34,200,1318,9354,69864,541323,4310950,35066384, \ldots$ &
	\OEISs{A307696}\\
	$2$ & $4$ &
	$4, 34, 368, 4685, 66416, 1013268, 16279788, 271594611, 4660794200,  \ldots$ &
	\OEISs{A307941}\\
	$3$ & $8$ &
	$8, 148, 3376, 89390, 2624872, 82866636, 2755019736, 95135709027, \ldots$ &
	\OEISs{A307942}\\
	$4$ & $16$ &
	$16, 616, 28832, 1556780, 93017264, 5971377672, 403667945712, \ldots$ &
	\OEISs{A307943}\\
	\bottomrule
\end{tabular}
\caption{$\DL$-history counting sequences of the complete species trees $\CB_h$ with $k=2^h$ leaves.
}
\label{tab:OEISComplete}
\end{table}
Applying Theorem~\ref{thm:DLasymgen} gives the asymptotic expansion of the coefficients for $n \to \infty$ as
\begin{align*}
	\cbcf_{h,n} = \beta_{h} \frac{\mu_h^{-n}}{n^{3/2}}\left(1 + \LandauO\left(\frac{1}{n}\right)\right),
\end{align*}
where $\beta_h>0$ and $\mu_h>0$ are nonnegative constants computed as follows.
\begin{proposition}
\label{prop:comdomgrowth}
The dominant singularity of $\cbgf_h(z)$ is $\mu_h = \frac{1-q_{h}}{4}$, where 
\begin{align*}
	\begin{cases}
    	q_0 = 0, \\
        q_{h+1} = (3-\sqrt{5-q_{h}})^2 & \text{ for } h \geq 0.
    \end{cases}	
\end{align*}
Furthermore, $q_h$ and $\mu_h$ are algebraic numbers of degree $2^h$. 	
\end{proposition}

\begin{proof}
    As for the caterpillar tree, we need to analyze the nested radicals. 
    To make this structure visible, we again define
    \begin{align} 
        \label{eq:ctGFansatz}
    	\cbgf_h(z) &= \frac{1 - \sqrt{Q_h(z)}}{2}.
    \end{align}
    Then, the radicands satisfy the following recurrence
    \begin{align}
    	\label{eq:comQk}
    	\begin{cases}
        	Q_0(z) = 1-4z, \\
        	Q_{h+1}(z) = -4 + 6 \sqrt{Q_h(z)} - Q_h(z), &\text{ for } h \geq 0.
        \end{cases}
    \end{align}
    When comparing it with the recurrence of radicands for the caterpillar grammar in~\eqref{eq:catRk} we notice a major difference: the coefficients are independent of $z$. 

	Then, the reasoning follows the same lines as the proof of Proposition~\ref{prop:catdomgrowth}. 
    Yet, due to the independence of the coefficients of $z$, the induction yields an explicit expression.
    Note that $Q_{h-i}(\mu_h) = q_i$.
    \hfill $\qed$
\end{proof}

In a similar way we are also able to compute the constant $\beta_h$ explicitly.

\begin{proposition}
    \label{prop:comleadconst}
    Using the notation of Proposition~\ref{prop:comdomgrowth}, the constant $\beta_h$ is equal to
    \begin{align*}
        \beta_h &= \sqrt{\frac{\mu_h}{16\pi} \prod_{i=1}^{h-1} \left( \frac{3}{q_i^2} -1 \right)}.
    \end{align*}
\end{proposition}

\begin{proof}
    By Equation~\eqref{eq:ctGFansatz} the singularity of $\cbgf_h(z)$ is determined by the smallest root $\mu_h$ of $Q_h(z)$. 
    The constant is determined by the expansion for $z\to \mu_h$:
    \begin{align*}
        Q_h(z) &= b_h(z-\mu_h) + \LandauO\left( (z-\mu_h)^2\right).
    \end{align*}
    By the recursive definition, $Q_h(z)$ is differentiable in $(0,\mu_{h-1})$ due to $\mu_{h} < \mu_{h-1}$.
    Thus, $b_h = Q_h'(\mu_h)$ is well-defined. 
    Differentiating the recurrence of $Q_h(z)$ we get
    \begin{align*}
        Q_h'(z) &= \left(\frac{3}{\sqrt{Q_{h-1}(z)}} - 1 \right) Q_{h-1}'(z).
    \end{align*}
    Iterating this relation and applying $Q_{h-i}(\mu_h) = q_i$ proves the claim.
    \hfill $\qed$
\end{proof}

As before, we computed the first few explicit values for the constant $\beta_h$ and the asymptotic growth factor $\mu_h^{-1}$, where $h$ is a power of $2$, and show them in Table~\ref{tab:catasygrowth}.




%% file: experiments.tex
\subsection{Empirical investigations and open questions}
\label{ssec:experiments}

In this section we present empirical results and observations derived using the counting and sampling algorithms described in Section~\ref{ssec:algorithms}. These results provide the first detailed view, especially in the $\DL$-model, of the general question: in how many ways can $n$ genes have evolved from a single ancestral gene, for a given species tree?

\subsubsection{Counting histories for random species trees} 
We are first interested in computing the number of histories in a given evolutionary model. We considered the following models: $\DL$-histories with an unranked or ranked species tree (called respectively models $\UDL$ and $\RDL$ from now), $\DLT$-histories with an unranked species tree or a ranked species tree (called respectively models $\UDLT$ and $\RDLT$ from now).

For a given evolutionary model and species tree $S$ of size $k$, let $h_S(n)$ be the number of histories of size $n$. As shown in Equation~(\ref{eq:catasymgen}) for the $\UDL$-model, this number grows asymptotically with $n$ as follows
$$
h_S(n) \simeq \gamma_S \frac{\rho_S^{-n}}{n^{3/2}}\left(1+\LandauO\left(\frac{1}{n}\right)\right)
$$
where $\gamma_S$ and $\rho_S$, both depend only on $S$. From now, we denote $\alpha_S = \rho_S^{-1}$  the \textit{exponential growth factor} for the number $h_S(n)$. In the $\UDL$-model, as discussed in Section~\ref{ssec:asymptotics}, we can compute precisely the growth factor from the grammar specifying the $\DL$-histories for the given species tree $S$. For other models, we can estimate $\lambda_S$ from the number $h_S(n)$ of histories of size $n$ as follows:
\begin{equation}
    \label{eq:as_growth_estimated}
    \alpha_S \simeq \frac{h_S(n)}{h_S(n-1)},
\end{equation}
this estimate precision increasing naturally with $n$.

\paragraph{\DL{}-models.} 
We considered species trees of size ranging from $k=3$ to $k=25$ and for each species tree size $k$, we generated $98$ random species tree of size $k$ under the uniform distribution, using the RANRUT algorithm described in~\cite{comb/NijenhuisW78}, and we completed this set of species tree by adding the caterpillar species tree with $k$ leaves and the balanced tree with $k$ leaves\footnote{Note that for a given $k$, any two balanced ordered binary trees with $k$ leaves differ only by swapping the left and right children of some internal nodes, so for our purpose there is essentially a unique balanced species tree for every value of $k$.}; so for small values of $k$, the same species tree can occur several times in the sample of $100$ trees. When working in the $\RDLT$-model, we generated, for each species tree $10$ random rankings under the uniform distribution, using the algorithm described in~\cite{comb/BodiniGGG18}. Then, for each instance, we computed the number of histories of size $n=50$ in the models $\UDL$, $\UDLT$ and $\RDLT$\footnote{We omit here the results for the $\RDL$-model as they are very similar to the results for the $\UDL$-model, with a lower dispersion.} and used these numbers to estimate the growth factor using~\eqref{eq:as_growth_estimated}.

Figure~\ref{fig:UDL_growth_constant} shows the exponential growth factor in the $\UDL$-model obtained using the exact approach described in Section~\ref{ssec:asymptotics} and the ratio between this exact growth factor and the growth factor estimated using the experimental approach described above. A first observation from Figure~\ref{fig:UDL_growth_constant} is that estimating the growth factor from the number of histories of size $n=50$ approximates well the exact growth factor in the $\UDL$-model; we believe it is also the case in the other models (data not shown).  

\begin{figure}[htbp]
    \centering
    \includegraphics[width=\textwidth]{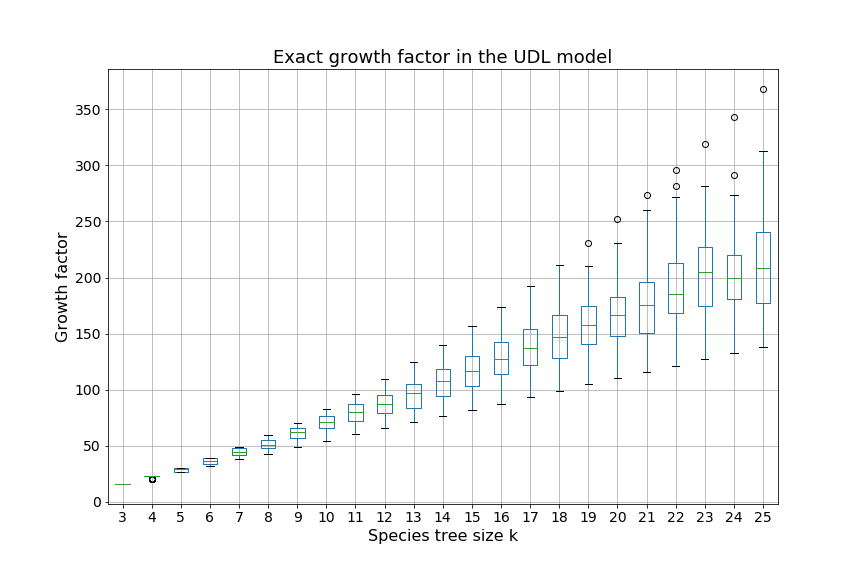}
    \includegraphics[width=\textwidth]{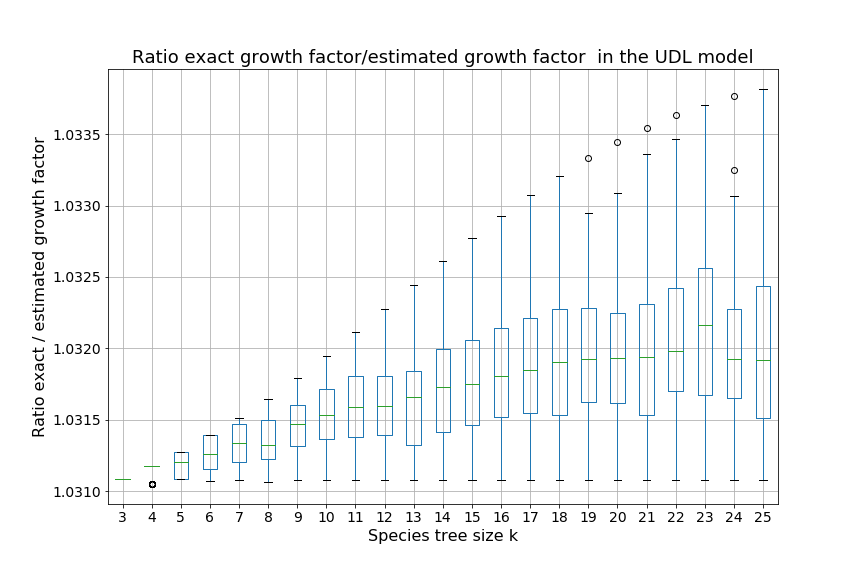}
    \caption{Box-plot of the distribution of the growth factor for each $100$ random species tree per size $k$ in the $\UDL$-model. (Top) Exact growth factor; (Bottom) Box-plot of the distribution, for each species tree, of the ratio between the exact growth factor and the estimated growth factor.}
    \label{fig:UDL_growth_constant}
\end{figure}

Moreover, following up on the results shown in Table~\ref{tab:catasygrowth}, our experiments lead to the following conjecture, characterizing the species trees leading to extreme growth factors for a given value of $k$.

\begin{conjecture}
    \label{conj:caterpillar-complete}
    For a given $k$, and $n$ large enough, the unranked species tree of size $k$ having the largest number of $\DL$-histories of size $n$ is the caterpillar tree; moreover the exponential growth factor of the number of histories for a caterpillar of size $k$ grows superlinearly as a function of $k$. Species trees having the smallest number of $\DL$-histories are balanced species trees of size $k$ and the exponential growth factor of the number of histories for a balanced tree of size $k$ grows linearly as a function of $k$.
\end{conjecture}

We verified that the conjecture is true for all values of $k$ in our experiments. We investigated several proof ideas, in particular linking the exponential growth factor to the number of unique subtrees in a species tree. Indeed this is a feature  for which caterpillar and balanced trees reach extreme values for a given value of $k$; actually the caterpillar is the unique tree with the maximum number of subtrees, while balanced trees have the minimum number of subtrees, although if $k$ is not a power of $2$, some unbalanced trees can have the same number of subtrees than balanced ones. We did find examples of pairs of species trees for which the one with the larger (resp. smaller) number of unique subtrees has a smaller (resp. larger) exponential growth factor. There are also species trees with the same number of unique subtrees than balanced trees of the same size and showing a larger exponential growth rate. So the number of unique subtrees is not the determinant leading to an extreme growth factor. We observed similar examples when considering the height of the species tree, another feature for which caterpillar and balanced trees attain extreme values. Generally the question of understanding which features of species trees of the same size that makes one having more \DL-histories than the other one is open.


\paragraph{\DLT{}-models.}
Next, we consider models including HGT; in Figure~\ref{fig:DLT_growth_constant}  we show the estimated growth constants in the $\UDLT$- and $\RDLT$-models. 

\begin{figure}[htbp]
    \centering
    \includegraphics[width=\textwidth]{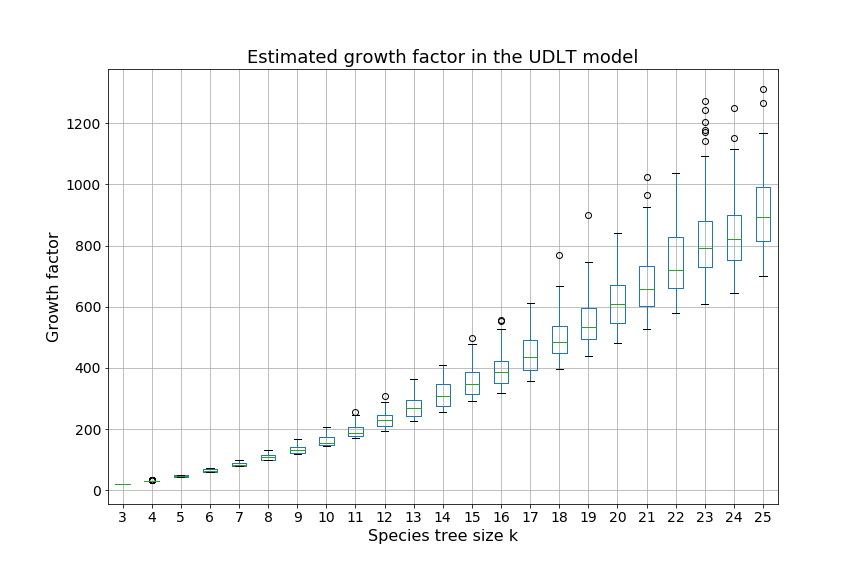}
    \includegraphics[width=\textwidth]{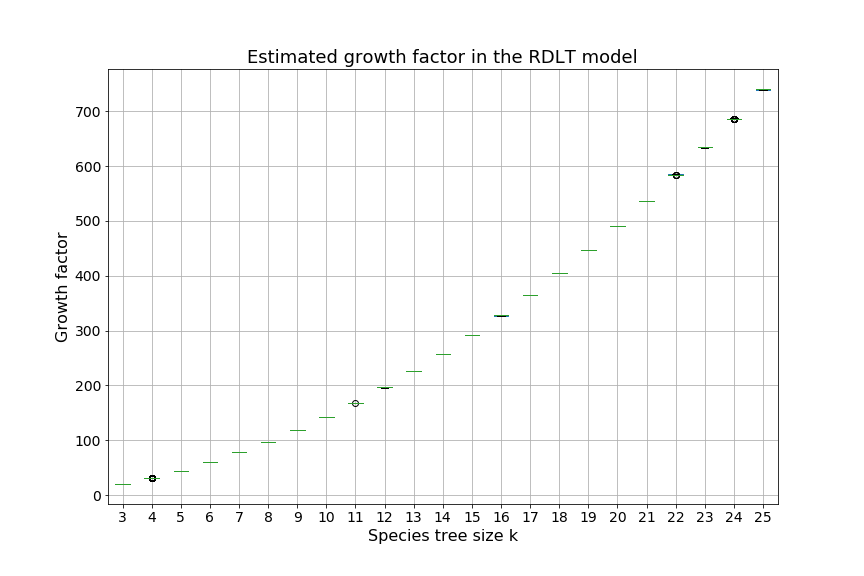}
    \caption{Box-plot of the distribution of the growth factor for each $100$ random species tree per size $k$ in the $\UDLT$ (Top) and $\RDLT$ (Bottom) models. The growth factor is estimated from the number of $\DLT$-histories of size $n=50$ using formula~(\ref{eq:as_growth_estimated}).}
    \label{fig:DLT_growth_constant}
\end{figure}

An observation that addresses one of the main questions motivating our work, is that the number of histories in models involving HGT grows much faster than in models excluding HGT; this is apparent by comparing the growth factors in the $\UDL$ and $\UDLT$ models, but even more through Figure~\ref{fig:ratio_DLT_DL} that shows the ratio of the number of $\DLT$-histories over the number of $\DL$-histories for selected pairs $(k,n)$, considered over all randomly chosen ranked or unranked species trees. We can observe that the ratios grow as large as $10^{40}$ in the unranked model and $10^{29}$ in the ranked model for histories of size $50$ over a species tree of size $25$, that correspond to parameters of realistic phylogenomics datasets. It is nevertheless interesting to observe that considering ranked species trees tames significantly the magnitude of the search space explosion when introducing HGT in a model.

\begin{figure}[htbp]
    \centering
    \includegraphics[width=\textwidth]{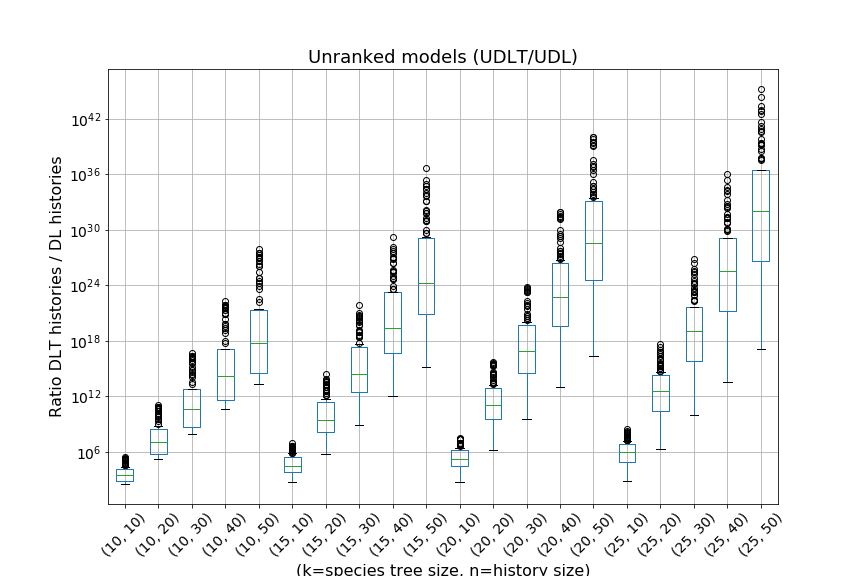}
    \includegraphics[width=\textwidth]{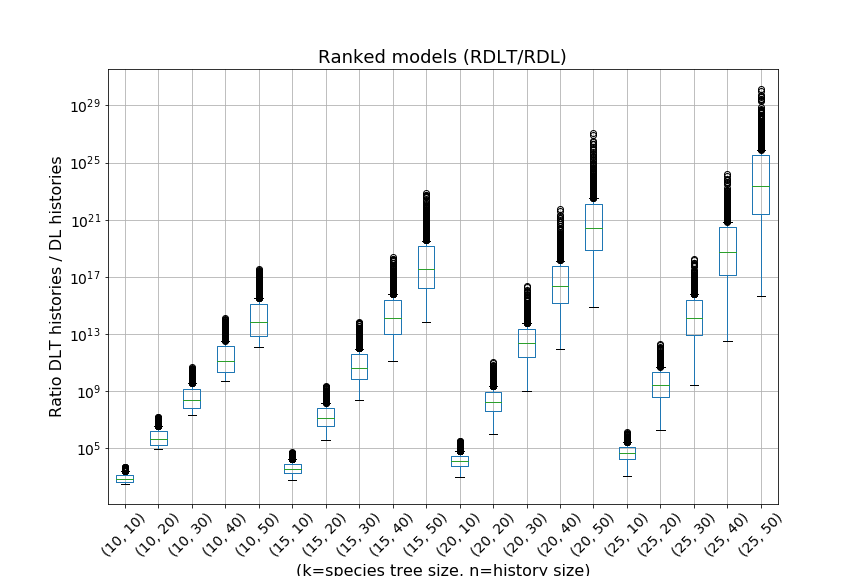}
    \caption{Box-plots of the distribution of the ratio of the number of $DLT$-histories over the number of $\DL$-histories over all species trees size $k$ and histories size $n$ for selected pairs $(k,n)$. The distributions are obtained, for each $(k,n)$, over  $100$ randomly chosen  (resp. $1000$)  unranked (resp. ranked) species trees. }
    \label{fig:ratio_DLT_DL}
\end{figure}

Finally, we can observe that in the $\RDLT$-model, the growth factor seems to be almost independent of the topology of the chosen species tree and ranking (Figure~\ref{fig:DLT_growth_constant} (Bottom)). Intuitively, this can be explained by the fact that a ranked species tree can almost be seen as a sequence of time slices, each composed of a set of branches (from $1$ branch for the time slice containing the root of $S$ to $k$ branches for the time slice containing all leaves), with exactly one ending with a speciation node while all other end by a unary node. Within each time slice, the genes can evolve freely by duplication and HGT, where a duplication can be seen as equivalent to a HGT within the same branch. Thus, the number of histories is dominated by the number of evolutionary events taking place in each time slice, with some variability being introduced by the number of genes leaving a time slice right after the only speciation node it contains, that can create extra gene copies entering the next time slice. 

In order to understand this phenomenon, we investigated a reduced evolutionary model, in which every speciation is followed by a random loss, i.e. does not create an extra gene copy entering the next time slice; we name this model the \RDTSL{}-model, where \SL{} stands for \textit{Speciation-Loss}. In this model, we are able to prove the independence of the chosen species trees.

\begin{theorem}
    \label{theo:RDLT-SL}
    In the \RDTSL{}-model, the number of histories of size $n$ is the same for every ranked species tree of size $k$.
\end{theorem}

\begin{proof}
    Let a ranked species tree of size $k$ be given, and consider the unary-binary tree induced by its time slices.
    We then transform this tree into a directed graph called the \emph{events graph} describing the possible events of duplication, HGT, and speciation in the following way: 
    \begin{enumerate}
    \item Label the leaves from $1$ to $k$.
    \item Label each internal node with a set containing the labels of the leaves of its induced subtree. These labels are the possible leaves reachable by speciation;
    \item Encode speciation events by super edges called \emph{speciation edges} which consist of the one (unary) or two (binary) edges leading to the children of a node.
    By doing so, the two edges are treated as a single edge;
    \item Encode duplication events by adding loops called \emph{duplication edges} to each node;
    \item Encode HGT events by adding edges called \emph{transfer edges} from each node to each other node within the same time slice;
    \end{enumerate}
    An example of this transformation is shown in Figure~\ref{fig:rDLT-SL-Trafo}.
    
    Let us briefly state some properties of the events graph. 
    The labels of the nodes of each time slice form a set partition of $\{1,\ldots,k\}$ by construction.
    Due to the rankings, each time slice contains one node more than the previous one and every path from the root to the previous leaves contains $k-1$ speciation edges.
    
    The main idea of the proof is that we can encode an history $H$ for a species tree $S$ of size $k$ by an ordered  unary-binary tree $H_e$ whose nodes are labeled by nodes of the events graph, that encodes unambiguously $H$, and then show that in the \RDTSL{}-model, given the events graph $E'$ of another ranked species tree $S'$ of the same size, we can transform $H_e$ into an ordered  unary-binary tree $H'_e$ whose nodes are labeled by nodes of $E'$ that encodes a unique history for $S'$. This establishes a one-to-one correspondence between the sets of histories for two arbitrary ranked species trees of size $k$, $S$ and $S'$, and thus proves the stated result.
    
    The principle of the encoding is to associate each internal node of a history with a (deterministic) label which is a node of the events graph. 
    Let $E$ be the events graph of $S$.
    The encoding works as follows: for a node $x$ of a history $H$ for species tree $S$, if $t$ is the time slice it belongs to and $i$ its left-most leaf (defined in a depth-first traversal of the ordered tree representing the history), then we label $x$ by the unique node of $E$ in the time slice $t$ that contains $i$. 
    Extant leaves stay labeled by their extant species.
    After deleting leaves corresponding to gene losses from the history, speciation-loss nodes become unary, while duplication and HGT nodes stay binary. 
    Call $H_e$ the ordered unary-binary tree for history $H$.
    The original history $H$ can be unambiguously recovered from $H_e$ and $E$, by reinserting these losses and removing the labels, as any edge of $H_e$ corresponds to an edge of $E$, so defines an evolutionary event. 
    
    Next, let $S'$ be another ranked species tree of the same size $k$ as $S$ and $E'$ its events graph.
    We transform $H_e$ into $H'_e$ as follows: for every node $x$, whose left-most leaf is $u$ and that belongs to time slice $t$, replace its label by the unique node of time slice $t$ of $E'$ that contains the $u$.
    This is always possible, as, by construction of the events graph in models with HGT, any leaf is reachable from any node.
    We claim that $H'_e$ defines unambiguously a history for $S'$. 
    The key argument to prove this claim is that, by the way we constructed $E'$ and $H'_e$, for any edge in $H'_e$ the labels of its two nodes, that are either in the same time slice or in consecutive time slices, are incident in $E'$: if both nodes are in the same time slice, then by construction of $E'$ they are either the same node (so linked by a duplication edge) or are incident by a transfer edge, while if they are in consecutive time slices, they contain a common species and so are incident by a speciation edge.
    It follows that $H'_e$ encodes a history $H'$ for $S'$.
    The construction from $H$ to $H'$ is deterministic and reversible, which provides a one-to-one correspondence between the histories of $S$ and the histories of $S'$ in the \RDTSL{}-model.
    
    Note that this construction does not work in models with no duplication, HGT or unrestricted speciation as the key argument that any edge in $H'_e$ can be found in $E'$ does not hold anymore, thus preventing to be able to transform $H'_e$ into a history for $S'$. 
\end{proof}

\begin{figure}[htbp]
    \centering
    \includegraphics[width=\textwidth]{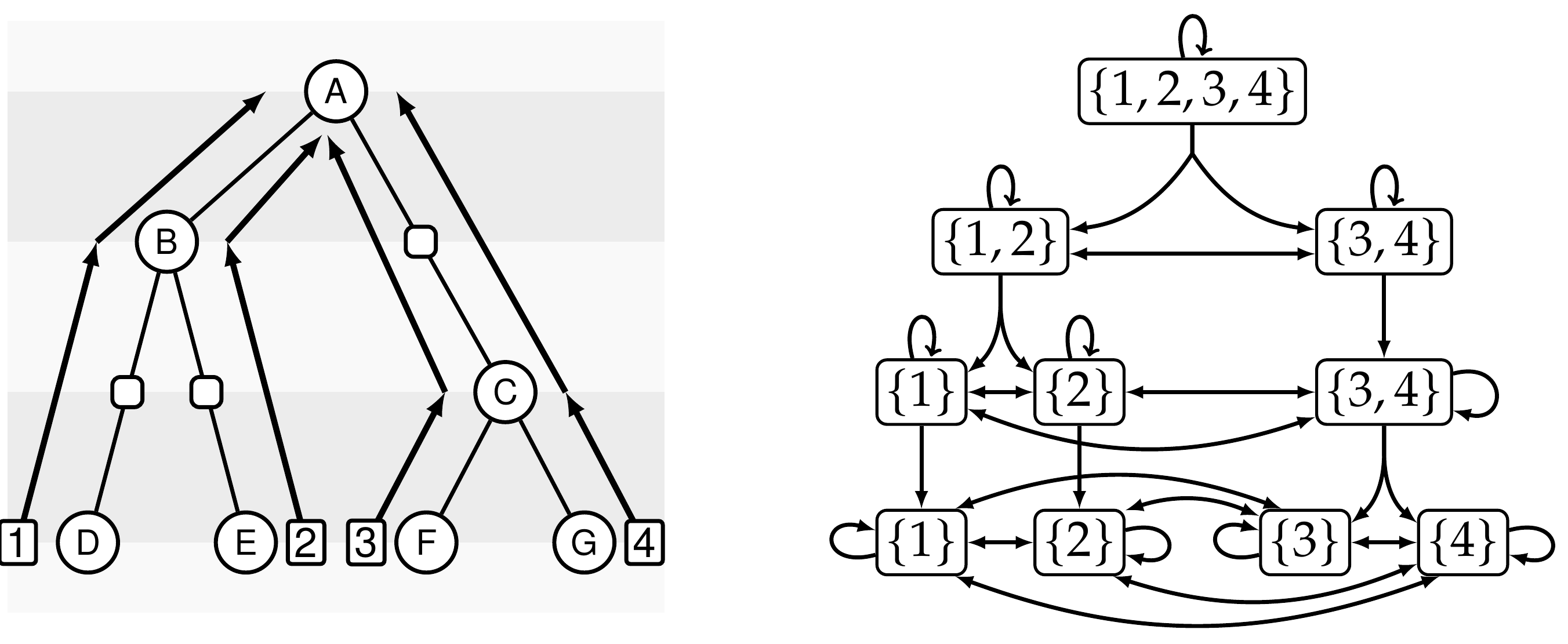}
    \caption{Transformation of a ranked species tree (left) ($\pi(A)=1,\pi(B)=2,\pi(C)=3,\pi(D)=\pi(E)=\pi(F)=\pi(G)=4$) into an events graph (right) in the \RDTSL{}-model used in the proof of Theorem~\ref{theo:RDLT-SL}.}
    \label{fig:rDLT-SL-Trafo}
\end{figure}

\begin{remark}
    From the previous proof we can also deduce an iterative tree growing algorithm for the histories offering an alternative explanation for Theorem~\ref{theo:RDLT-SL}. 
    Every internal node gets a label that is a pair consisting of its time slice and the number of its left-most leaf. 
    Note that this uniquely identifies a node in the species tree. 
    
    We start with a root node labeled by the first time slice and an arbitrary number from $\{1,\ldots,k\}$.
    At every step, choose a leaf of the current history and consider the corresponding node in the events graph. 
    Then traverse one of its edges and perform the action of this edge:
    If it is a speciation edge then add a new node with a label consisting of the successive time slice and the same number as only child
    .
    If it is a duplication or transfer edge then add a left child with the same label as the root and a right child labeled with the current time slice and an arbitrary number from the set the edge is pointing to.
    Once all leaves correspond to extant nodes the tree is a valid history.
    %
\end{remark}
 
\begin{remark}
    The construction of the events graph in Theorem~\ref{theo:RDLT-SL} can be adapted to all models. If there are no duplication events, the duplication edges are removed; if there are no HGT events, the transfer edges are removed. 
    The characteristics of the $\SL{}$ dynamics are not encoded in the events graph but in the bijection or the history growing algorithm.
\end{remark}


\subsubsection{On the parsimony and profile of random histories.}
We also considered at the distribution of the evolutionary score for randomly sampled histories, where the score of a history is the sum of the number of duplications, losses and HGT, for $k=16$ and $n=30$, over $50$ random unranked species trees, sampling $10,000$ random histories for each species tree. 
 
Figure~\ref{fig:scores} below suggests that the space of histories for a given species tree is dominated by histories with a relatively high score and that, as expected, for a given species tree including HGT in the evolutionary model leads to a significant decrease of the evolutionary score of histories. 

\begin{figure}[htbp]
     \centering
     \includegraphics[width=\textwidth]{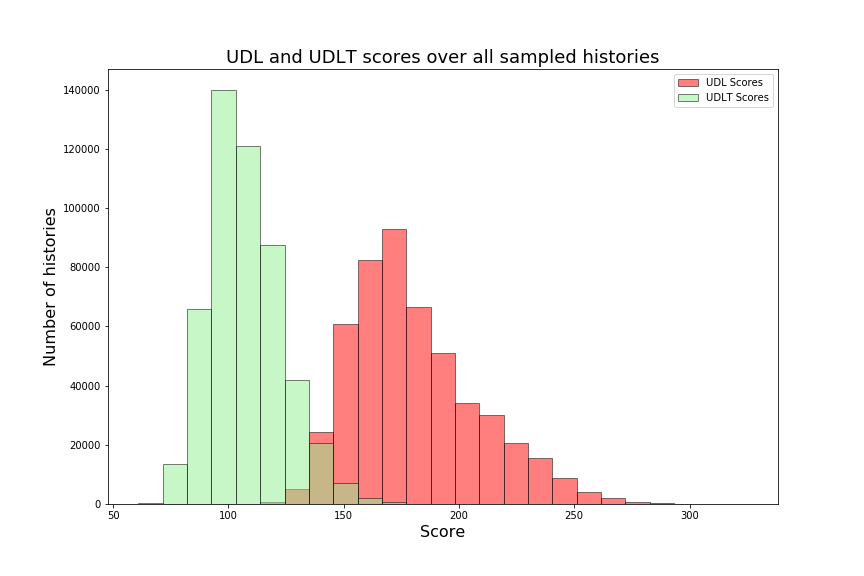}
     \caption{Distribution of the score (number of duplications plus losses plus HGT) over $50$ random species trees of size $16$ and $10,000$ random histories of size $30$ per tree in the \UDL{}- and \UDLT{}-models.}
      \label{fig:scores}
\end{figure}
     
In fact, when looking at the distribution of the number of duplications in the \UDLT-model (results not shown), we observed that the duplication number drops significantly in the $\UDLT$-model compared to the $\UDL$-model. We can also note that, when comparing the score of histories in the $\UDL$-model and the number of duplications, most of the score is  due to gene losses (Figure~\ref{fig:dist_events}), a characteristic we also see in the $\UDLT$-model where the number of duplications (resp. HGT)  exceeds rarely $5$ (resp. $25$) in the sampled histories.

\begin{figure}[htbp]
     \centering
     \includegraphics[width=\textwidth]{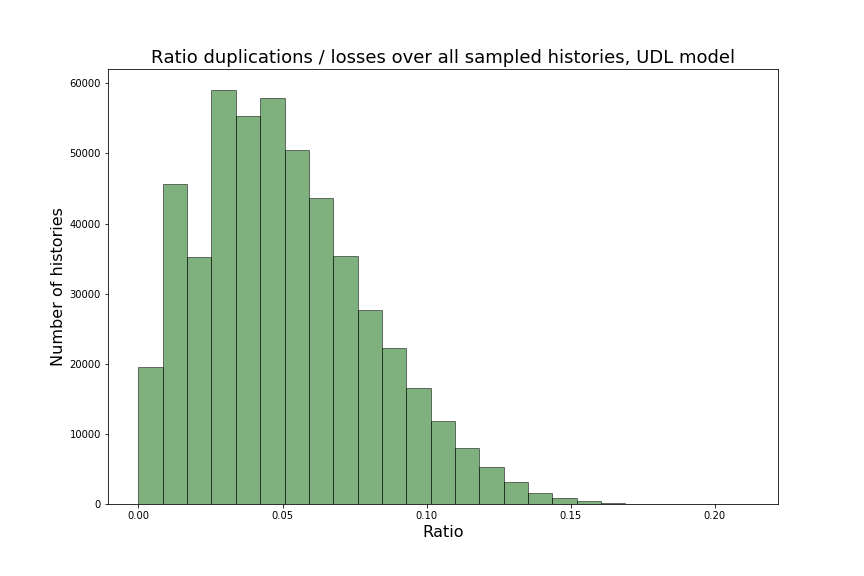}
     \includegraphics[width=\textwidth]{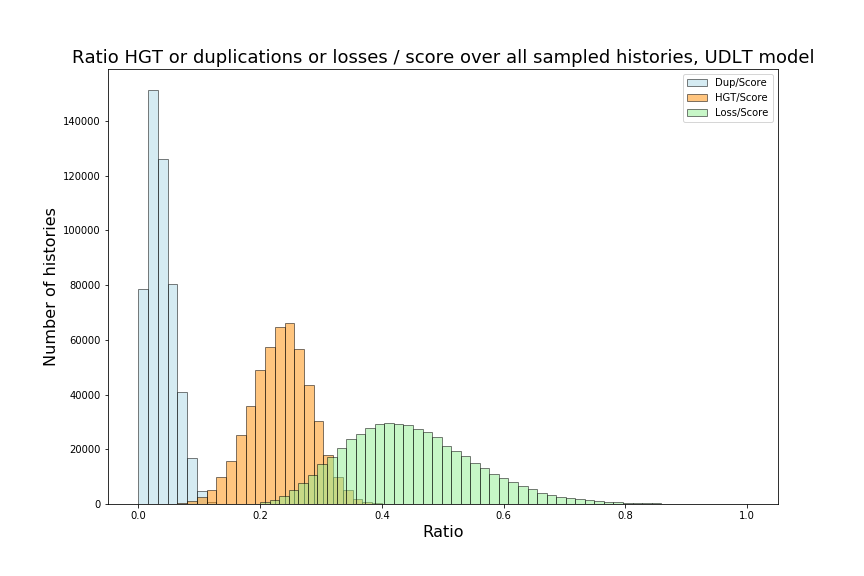}
     \caption{Distribution of the ratio Duplications / Losses  in the \UDL{} (Top) and of the ratios HGT / score, Duplications/score and Losses/score in the \UDLT{}-model (Bottom). For both figures the distribution is over  $50$ random species trees of size $16$ and $10,000$ random histories of size $30$ per tree.}
      \label{fig:dist_events}
\end{figure}

%% file: conclusion.tex
\section{Conclusion and perspectives}
\label{sec:conclusion}

Our work introduces the first results on counting and sampling evolutionary scenarios in models accounting for gene duplication, gene loss and HGT. The originality of our work, compared to previous work in the reconciliation framework, is that we only consider the species tree to be given, and thus consider all possible evolutionary histories of a given size, i.e. leading to a given number of genes. Our results include formal grammars describing this combinatorial space, together with counting and sampling algorithms, obtained using either dynamic programming or enumerative and analytic combinatorics methods. These results complement a growing body of work developed over the last few years in the case of matching gene and species trees.

Using our method, we were able to obtain precise asymptotics on the number of histories for the two specific species trees, the rooted caterpillar and the complete binary tree in the unranked $\DL$-model, although our method also applies to any given species tree in this model. Our counting and sampling algorithms allowed us to complement these results for other models, especially models accounting for HGT. Our experimental results provide a first global view of the space of potential evolutionary histories for a given species tree. They confirm the expected fact that introducing HGT in a model result in a dramatic increase of the space of possible histories; they also lead to the interesting observation that in the ranked $\DLT$-model, the total number of histories is asymptotically almost independent of the given species tree. 

Our work suggests several avenues for further research. First, our notion of evolutionary history assumes that gene trees are ordered, i.e. that gene copies created by a gene duplication are distinguishable; this differs from the notion of reconciled gene trees, where duplicated copies are not distinguishable. While our assumption follows naturally from an evolutionary biology point of view, it would be interesting to see if our approach could be applied to count and sample reconciliations instead of histories. Next, the last few years have seen the development of more comprehensive models of gene family evolution, accounting for example for genes appearing at a given species by an HGT from an unsampled or extinct species~\cite{sysbio/SzollosiTLD13}, incomplete lineage sorting (ILS)~\cite{gr/RasmussenK12,gr/WuRK14,isbra/ZhangW17,bcb/DuN18,bioinformatics/StolzerLXSVD12,jtb/ChanRS17}, or gene conversion~\cite{jmb/HasicT19}. In these models, reconciled gene trees can be computed using dynamic programming algorithms and it is natural to ask if such algorithms could be turned into grammars for the corresponding space of evolutionary scenarios. 
Last, from an applied point of view, a limitation of our work lies in the fact that histories are parameterized by their size, i.e. the number of extant genes, while in applications, the genes of a gene family are assigned to specific extant species. Ideally, in order to explore (through counting or sampling) the space of all possible evolutionary scenarios for a gene families whose distribution of genes in extant species is given, we would need to parameterize our algorithms by this distribution, which leads to dynamic programming algorithms with a much higher time and space complexity, dependent on the number of extant species. However, we believe that advanced combinatorial sampling, especially  multiparametric combinatorial samplers~\cite{bodini:hal-00450763,analco/BendkowskiBD18}, can be used within the framework we developed in the present work  to provide efficient counting and sampling algorithms.

\bigskip
\noindent\textbf{Funding:}
The first author is supported by a Discovery Grant of the Natural Sciences and Engineering Research Council of Canada (RGPIN-2017-03986). This research was enabled in part by support provided by Westgrid (\url{https://www.westgrid.ca/}) and Compute Canada (\url{https://www.computecanada.ca}) through a Resource Allocation (ID 838) to the first author.
The third author was supported by the Exzellenzstipendium of the Austrian Federal Ministry of Education, Science and Research and the Erwin Schr{\"o}dinger Fellowship of the Austrian Science Fund (FWF):~J~4162-N35. 

\bigskip
\noindent\textbf{Conflict of interest:}
The authors declare that they have no conflict of interest.

%% file: counting_DLT-Histories_arxiv.bbl
\begin{thebibliography}{10}

\bibitem{pnas/AkerborgSAL09}
{\"O}.~{\r A}kerborg, B.~Sennblad, L.~Arvestad, and J.~Lagergren.
\newblock Simultaneous bayesian gene tree reconstruction and reconciliation
  analysis.
\newblock {\em Proceedings of the National Academy of Sciences},
  106(14):5714--5719, 2009.

\bibitem{jacm/ArvestadLS09}
L.~Arvestad, J.~Lagergren, and B.~Sennblad.
\newblock The gene evolution model and computing its associated probabilities.
\newblock {\em Journal of the {ACM}}, 56(2):7:1--7:44, 2009.

\bibitem{jtb/ChanRS17}
Y.~ban Chan, V.~Ranwez, and C.~Scornavacca.
\newblock Inferring incomplete lineage sorting, duplications, transfers and
  losses with reconciliations.
\newblock {\em Journal of Theoretical Biology}, 432:1--13, 2017.

\bibitem{comb/BanderierW17}
C.~Banderier and M.~Wallner.
\newblock Lattice paths with catastrophes.
\newblock {\em {Discrete Mathematics \& Theoretical Computer Science}}, {Vol 19
  no. 1}, Sept. 2017.
\newblock Full version of extended abstract with the same title appeared in the
  Proceedings of conference on Random Generation of Combinatorial Structures --
  \{GASCom\} 2016.

\bibitem{cmb/BansalAK13}
M.~S. Bansal, E.~J. Alm, and M.~Kellis.
\newblock Reconciliation revisited: Handling multiple optima when reconciling
  with duplication, transfer, and loss.
\newblock {\em Journal of Computational Biology}, 20(10):738--754, 2013.

\bibitem{bioinformatics/BansalKKK18}
M.~S. Bansal, M.~Kellis, M.~Kordi, and S.~Kundu.
\newblock Ranger-dtl 2.0: rigorous reconstruction of gene-family evolution by
  duplication, transfer and loss.
\newblock {\em Bioinformatics}, 34(18):3214--3216, 2018.

\bibitem{analco/BendkowskiBD18}
M.~Bendkowski, O.~Bodini, and S.~Dovgal.
\newblock Polynomial tuning of multiparametric combinatorial samplers.
\newblock In {\em Proceedings of the Fifteenth Workshop on Analytic
  Algorithmics and Combinatorics, {ANALCO} 2018, New Orleans, LA, USA, January
  8-9, 2018.}, pages 92--106. {SIAM}, 2018.

\bibitem{comb/BodiniGGG18}
O.~Bodini, D.~Gardy, B.~Gittenberger, and Z.~Go{\l}{\k{e}}biewski.
\newblock On the number of unary-binary tree-like structures with restrictions
  on the unary height.
\newblock {\em Annals of Combinatorics}, 22(1):45--91, 2018.

\bibitem{bodini:hal-00450763}
O.~Bodini and Y.~Ponty.
\newblock {Multi-dimensional Boltzmann Sampling of Languages}.
\newblock In M.~Drmota and B.~Gittenberger, editors, {\em {21st International
  Meeting on Probabilistic, Combinatorial, and Asymptotic Methods in the
  Analysis of Algorithms (AofA'10)}}, volume DMTCS Proceedings vol. AM, 21st
  International Meeting on Probabilistic, Combinatorial, and Asymptotic Methods
  in the Analysis of Algorithms (AofA'10) of {\em DMTCS Proceedings}, pages
  49--64, Vienna, Austria, June 2010. {Discrete Mathematics and Theoretical
  Computer Science}.

\bibitem{comb/BonaF09}
M.~Bóna and P.~Flajolet.
\newblock Isomorphism and symmetries in random phylogenetic trees.
\newblock {\em Journal of Applied Probability}, 46(4):1005–1019, 2009.

\bibitem{tree/DegnanR09}
J.~Degnan and N.~Rosenberg.
\newblock Gene tree discordance, phylogenetic and the multispecies coalescent.
\newblock {\em Trends in Ecology \& Evolution}, 24:332--340, 2009.

\bibitem{evolution/DegnanS05}
J.~H. Degnan and L.~A. Salter.
\newblock Gene tree distribution under the coalescent process.
\newblock {\em Evolution}, 59(1):24--37, 2005.

\bibitem{jcb/DisantoR15}
F.~Disanto and N.~A. Rosenberg.
\newblock Coalescent histories for lodgepole species trees.
\newblock {\em Journal of Computational Biology}, 22(10):918--929, 2015.

\bibitem{tcbb/DisantoR16}
F.~Disanto and N.~A. Rosenberg.
\newblock Asymptotic properties of the number of matching coalescent histories
  for caterpillar-like families of species trees.
\newblock {\em {IEEE/ACM} Transactions on Compututational Biology and
  Bioinformatics}, 13(5):913--925, 2016.

\bibitem{jcb/DisantoR17}
F.~Disanto and N.~A. Rosenberg.
\newblock Enumeration of ancestral configurations for matching gene trees and
  species trees.
\newblock {\em Journal of Computational Biology}, 24(9):831--850, 2017.

\bibitem{bmb/DisantoR17}
F.~Disanto and N.~A. Rosenberg.
\newblock On the number of non-equivalent ancestral configurations for matching
  gene trees and species trees.
\newblock {\em Bulletin of Mathematical Biology}, 2017.
\newblock in press.

\bibitem{jmb/DisantoR18}
F.~Disanto and N.~A. Rosenberg.
\newblock Enumeration of compact coalescent histories for matching gene trees
  and species trees.
\newblock {\em Journal of Mathematical Biology}, 2018.
\newblock in press.

\bibitem{cmb/DoyonCH09}
J.-P. Doyon, C.~Chauve, and S.~Hamel.
\newblock Space of gene/species trees reconciliations and parsimonious models.
\newblock {\em Journal of Computational Biology}, 16(10):1399--1418, 2009.

\bibitem{bib/DoyonRDB11}
J.-P. Doyon, V.~Ranwez, V.~Daubin, and V.~Berry.
\newblock Models, algorithms and programs for phylogeny reconciliation.
\newblock {\em Briefings in Bioinformatics}, 12(5):392--400, 2011.

\bibitem{comb/Drmota97}
M.~Drmota.
\newblock Systems of functional equations.
\newblock {\em Random Structures \& Algorithms}, 10(1‐2):103--124, 1997.

\bibitem{bcb/DuN18}
P.~Du and L.~Nakhleh.
\newblock Species tree and reconciliation estimation under a
  duplication-loss-coalescence model.
\newblock In {\em Proceedings of the 2018 ACM International Conference on
  Bioinformatics, Computational Biology, and Health Informatics}, BCB '18,
  pages 376--385. ACM, 2018.

\bibitem{cmb/DurandHV06}
D.~Durand, B.~V. Halldórsson, and B.~Vernot.
\newblock A hybrid micro–macroevolutionary approach to gene tree
  reconstruction.
\newblock {\em Journal of Computational Biology}, 13(2):320--335, 2006.

\bibitem{comb/FlajoletS09}
P.~Flajolet and R.~Sedgewick.
\newblock {\em Analytic Combinatorics}.
\newblock Cambridge University Press, 2009.

\bibitem{flajolet1994calculus}
P.~Flajolet, P.~Zimmermann, and B.~Van~Cutsem.
\newblock A calculus for the random generation of labelled combinatorial
  structures.
\newblock {\em Theoretical Computer Science}, 132(1-2):1--35, 1994.

\bibitem{comb/GittenbergerJW18}
B.~Gittenberger, E.~Y. Jin, and M.~Wallner.
\newblock On the shape of random {P}\'olya structures.
\newblock {\em Discrete Mathematics}, 341(4):896 -- 911, 2018.

\bibitem{syszoo/GoodmanCMRM79}
M.~Goodman, J.~Czelusniak, G.~W. Moore, A.~E. Romero-Herrera, and G.~Matsuda.
\newblock Fitting the gene lineage into its species lineage, a parsimony
  strategy illustrated by cladograms constructed from globin sequences.
\newblock {\em Systematic Zoology}, 28(2):132--163, 1979.

\bibitem{bmcbi/GoreckiBE11}
P.~G{\'o}recki, G.~J. Burleigh, and O.~Eulenstein.
\newblock Maximum likelihood models and algorithms for gene tree evolution with
  duplications and losses.
\newblock {\em BMC Bioinformatics}, 12(1):S15, 2011.

\bibitem{tcs/GoreckiT06}
P.~G\'{o}recki and J.~Tiuryn.
\newblock {DLS}-trees: A model of evolutionary scenarios.
\newblock {\em Theoretical Computer Science}, 359(1):378--399, 2006.

\bibitem{cmb/GoreckiE14}
P.~Górecki and O.~Eulenstein.
\newblock Drml: Probabilistic modeling of gene duplications.
\newblock {\em Journal of Computational Biology}, 21(1):89--98, 2014.

\bibitem{jmb/HasicT19}
D.~Hasi{\'{c}} and E.~Tannier.
\newblock Gene tree species tree reconciliation with gene conversion.
\newblock {\em Journal of Mathematical Biology}, 78(6):1981--2014, 2019.

\bibitem{bioinformatics/JacoxCSPS16}
E.~Jacox, C.~Chauve, G.~J. Szöllősi, Y.~Ponty, and C.~Scornavacca.
\newblock eccetera: comprehensive gene tree-species tree reconciliation using
  parsimony.
\newblock {\em Bioinformatics}, 32(13):2056--2058, 2016.

\bibitem{sysbio/Maddison97}
W.~P. Maddison.
\newblock Gene trees in species trees.
\newblock {\em Systematic Biology}, 46(3):523--536, 1997.

\bibitem{comb/NijenhuisW78}
A.~Nijenhuis and H.~S. Wilf.
\newblock {\em Combinatorial Algorithms}.
\newblock Academic Press, 1978.

\bibitem{cmb/OvadiaFCL11}
Y.~Ovadia, D.~Fielder, C.~Conow, and R.~Libeskind-Hadas.
\newblock The cophylogeny reconstruction problem is {NP}-complete.
\newblock {\em Journal of Computational Biology}, 18(1):59--65, 2011.

\bibitem{bioinformatics/Wu17}
J.~Pei and Y.~Wu.
\newblock {STELLS2}: fast and accurate coalescent-based maximum likelihood
  inference of species trees from gene tree topologies.
\newblock {\em Bioinformatics}, 33(12):1789--1797, 2017.

\bibitem{jmb/RanwezSDB16}
V.~Ranwez, C.~Scornavacca, J.-P. Doyon, and V.~Berry.
\newblock Inferring gene duplications, transfers and losses can be done in a
  discrete framework.
\newblock {\em Journal of Mathematical Biology}, 72(7):1811--1844, 2016.

\bibitem{gr/RasmussenK12}
M.~D. Rasmussen and M.~Kellis.
\newblock {Unified modeling of gene duplication, loss, and coalescence using a
  locus tree}.
\newblock {\em Genome Research}, 22(4):755--765, 2012.

\bibitem{jcb/Rosenberg07}
N.~A. Rosenberg.
\newblock Counting coalescent histories.
\newblock {\em Journal of Computational Biology}, 14(3):360--377, 2007.

\bibitem{bioinformatics/ScornavaccaJS15}
C.~Scornavacca, E.~Jacox, and G.~J. Sz{\"o}ll{\H{o}}si.
\newblock Joint amalgamation of most parsimonious reconciled gene trees.
\newblock {\em Bioinformatics}, 31(6):841--848, 2015.

\bibitem{sysbio/SjostrandTDASL14}
J.~Sjöstrand, A.~Tofigh, V.~Daubin, L.~Arvestad, B.~Sennblad, and
  J.~Lagergren.
\newblock A bayesian method for analyzing lateral gene transfer.
\newblock {\em Systematic Biology}, 63(3):409--420, 2014.

\bibitem{comb/Sloane}
N.~J.~A. Sloane.
\newblock The {O}n-{L}ine {E}ncyclopedia of {I}nteger {S}equences ({OEIS}).
\newblock \url{http://oeis.org}.

\bibitem{bioinformatics/StolzerLXSVD12}
M.~Stolzer, H.~Lai, M.~Xu, D.~Sathaye, B.~Vernot, and D.~Durand.
\newblock Inferring duplications, losses, transfers and incomplete lineage
  sorting with nonbinary species trees.
\newblock {\em Bioinformatics}, 28(18):i409--i415, 2012.

\bibitem{mmb/SzollosiD12}
G.~J. Sz{\"o}ll{\H{o}}si and V.~Daubin.
\newblock Modeling gene family evolution and reconciling phylogenetic discord.
\newblock In M.~Anisimova, editor, {\em Evolutionary Genomics: Statistical and
  Computational Methods, Volume 2}, volume 856 of {\em Methods in Molecular
  Biology}, pages 29--51. Humana Press, 2012.

\bibitem{sysbio/SzollosiRBTD13}
G.~J. Sz{\"o}ll{\H{o}}si, W.~Rosikiewicz, B.~Boussau, E.~Tannier, and
  V.~Daubin.
\newblock Efficient exploration of the space of reconciled gene trees.
\newblock {\em Systematic Biology}, 62(6):901--912, 2013.

\bibitem{sysbio/SzollosiTDB15}
G.~J. Sz{\"o}ll{\H{o}}si, E.~Tannier, V.~Daubin, and B.~Boussau.
\newblock The inference of gene trees with species trees.
\newblock {\em Systematic Biology}, 64(1):e42--e62, 2015.

\bibitem{sysbio/SzollosiTLD13}
G.~J. Szöllősi, E.~Tannier, N.~Lartillot, and V.~Daubin.
\newblock Lateral gene transfer from the dead.
\newblock {\em Systematic Biology}, 62(3):386--397, 2013.

\bibitem{tcbb/TofighHL11}
A.~Tofigh, M.~T. Hallett, and J.~Lagergren.
\newblock Simultaneous identification of duplications and lateral gene
  transfers.
\newblock {\em {IEEE/ACM} Transactions on Computational Biology and
  Bioinformatics}, 8(2):517--535, 2011.

\bibitem{wilf1977unified}
H.~S. Wilf.
\newblock A unified setting for sequencing, ranking, and selection algorithms
  for combinatorial objects.
\newblock {\em Advances in Mathematics}, 24(3):281--291, 1977.

\bibitem{evolution/Wu12}
Y.~Wu.
\newblock Coalescent-based species tree inference from gene tree topologies
  under incomplete lineage sorting by maximum likelihood.
\newblock {\em Evolution}, 66(3):763--775, 2012.

\bibitem{bioinformatics/Wu16}
Y.~Wu.
\newblock An algorithm for computing the gene tree probability under the
  multispecies coalescent and its application in the inference of population
  tree.
\newblock {\em Bioinformatics}, 32(12):i225--i233, 2016.

\bibitem{gr/WuRK14}
Y.-C. Wu, M.~D. Rasmussen, and M.~Kellis.
\newblock {Most parsimonious reconciliation in the presence of gene
  duplication, loss, and deep coalescence using labeled coalescent trees}.
\newblock {\em Genome Research}, 24(3):475--486, 2014.

\bibitem{isbra/ZhangW17}
B.~Zhang and Y.-C. Wu.
\newblock Coestimation of gene trees and reconciliations under a
  duplication-loss-coalescence model.
\newblock In Z.~Cai, O.~Daescu, and M.~Li, editors, {\em Bioinformatics
  Research and Applications}, volume 10330 of {\em Lecture Notes in Computer
  Science}, pages 196--210. Springer, 2017.

\end{thebibliography}
